\pdfoutput=1

\documentclass[reqno,11pt]{amsart}
\usepackage{glaudo_pkg_en}

\title[Expansion of the fundamental solution]{Expansion of the fundamental solution of a second order elliptic operator with analytic coefficients}

\author{Federico Franceschini}
\address{ETH, Rämistrasse 101, 8092 Zürich, Switzerland}
\email{federico.franceschini@math.ethz.ch}

\author{Federico Glaudo}
\address{ETH, Rämistrasse 101, 8092 Zürich, Switzerland}
\email{federico.glaudo@math.ethz.ch}

\begin{document} 
\begin{abstract}
    Let $L$ be a second-order elliptic operator with analytic coefficients defined in $B_1\subseteq\R^n$.
    We construct explicitly and canonically a fundamental solution for the operator, i.e., a function $u:B_{r_0}\to\R$ such that $Lu=\delta_0$. 
    
    As a consequence of our construction, we obtain an expansion of the fundamental solution in homogeneous terms (homogeneous polynomials divided by a power of $\abs{x}$, plus homogeneous polynomials multiplied by $\log(\abs{x})$ if the dimension $n$ is even) which improves the classical result of \cite{John1950}.
    The control we have on the \emph{complexity} of each homogeneous term is optimal and in particular, when $L$ is the Laplace-Beltrami operator of an analytic Riemannian manifold, we recover the construction of the fundamental solution due to Kodaira \cite{Kodaira49}.
    
    The main ingredients of the proof are a harmonic decomposition for singular functions and the reduction of the convergence of our construction to a nontrivial estimate on weighted paths on a graph with vertices indexed by $\Z^2$.
\end{abstract}

\maketitle
\section{Introduction}

For $n\geq 2$, consider the elliptic differential operator
\begin{equation*}
    L\defeq \sum_{1\le i,j\le n} A_{ij}\partial_{ij} + \sum_{1\le i\le n} b_i \partial_i + c,
\end{equation*}
where $A_{ij},b_i,c$ are analytic in the unit ball $B_1\subseteq \R^n$ and $A(x)$ is a positive symmetric matrix for every $x\in B_1$.

Given an open set $U\subseteq B_1$ which contains the origin, a function $u\in L^1_{loc}(U)$ is a \emph{fundamental solution} for the operator $L$ in $U$ if $Lu=\delta$ in $U$ in the distributional sense, where $\delta$ denotes the Dirac delta distribution. If $U$ is small enough, a fundamental solution always exists (for example see \cite[Sections 13.2-13.3]{2Hormander2005}).

Since $L$ is elliptic, any fundamental solution is singular only at the origin and studying this singularity is a classical problem. Let us point out that even if the fundamental solution is not unique (we did not fix the boundary conditions), the singularity at the origin is unique: if $u$ and $u'$ are fundamental solutions in $U$ and $U'$ respectively, then $u-u'$ is analytic in $U\cap U'$. In this sense the singularity of the fundamental solution is completely determined by the local behavior of $A,b,c$ at the origin.

Due to the expected nature of the singularity, it is natural to try to construct a fundamental solution $u$ with the following structure
\begin{equation}\label{eq:ansatz}
    u(x) = \frac{1}{\abs{x}^{n-2}}\sum_{\ell\ge 0}u_\ell + \log(\abs{x}) v(x),
\end{equation}
where, for all $\ell\ge 0$, $u_\ell$ is an $\ell$-homogeneous function and $v$ is a regular function. Notice that, if $L=\Delta$, then the standard fundamental solution can be expressed in the form \eqref{eq:ansatz} in every dimension $n\geq 2$.

The question now is: can we construct a legitimate fundamental solution having the form \eqref{eq:ansatz}? If so, what can one say about the functions $u_\ell$ and $v$? The work of John \cite{John1950} shows that the answer to the first question is positive and that the $u_\ell$ are analytic on the unit sphere, while $v(x)$ is real analytic.

\begin{theorem}[{{\cite[Eqs. (5.14), (5.20)]{John1950}}}]\label{thm:john}
    For $n\geq 2$, consider the elliptic differential operator
    \begin{equation*}
            L \defeq \sum_{1\le i,j\le n} A_{ij}\partial_{ij} + \sum_{1\le i\le n} b_i \partial_i + c,
    \end{equation*}
    where $A_{ij},b_i,c$ are analytic in the unit ball $B_1$ and $A(x)$ is a positive symmetric matrix for every $x\in B_1$.
    Then there is a radius $r_0=r_0(n, A, b, c)>0$ such that the following statement hold.
    
    There is an analytic function $v\colon B_{r_0}\to\R$ (which is null if the dimension $n$ is odd) and a sequence of homogeneous functions $(u_\ell)_{\ell\ge 0}\colon B_{r_0}\to\R$, $u_\ell$ being $\ell$-homogeneous, such that
    \begin{equation*}
        u(x) \defeq v(x)\log(\abs{x}) + \frac1{\abs{x}^{n-2}}\sum_{\ell\ge 0} u_\ell
    \end{equation*}
    converges absolutely in the $C^{\infty}$ topology\footnote{Notice that the function $u$ itself is not in $C^\infty$, but the differences between $u$ and the truncated sums converge to $0$ in the $C^k$-topology for any $k\ge 0$.} in $B_{r_0}$ and satisfies, in the distributional sense, $Lu=\delta$. Moreover, for each $\ell\ge0$, we have
    \begin{equation*}
        u_\ell(x) = \abs{x}^\ell \hat u_\ell\Big(\frac{x}{\abs{x}}\Big) ,
    \end{equation*}
    where each $\hat u_\ell$ is an analytic function on the unit sphere.
\end{theorem}
The construction of $u$ in \cite{John1950} relies on the Cauchy-Kowalewski theorem and on a technique going back to Hadamard (see \cite{Hadamard1904} and \cite[Section 17.4]{3Hormander2007} and \cite[Chapter III]{Shimakura92}). 
The method works for elliptic operators of any order.

In some relevant applications it is crucial to understand whether the expansion \cref{eq:ansatz} can be simplified even further. An important example is the following result of Kodaira.
\begin{theorem}[{{\cite[617]{Kodaira49}}}]\label{thm:kodaira}
    Let $(M,g)$ be an analytic Riemannian manifold of dimension $n\geq 2$ and let $\lapl_g$ be the Laplace-Beltrami operator. Let $u$ be defined around $x_0\in M$ and satisfy $\lapl_gu=\delta_{x_0}$. Then, in normal coordinates around $x_0$, we have
    \begin{equation*}
        u(x)=\frac{\hat u(x)}{\abs{x}^{n-2}}+v(x)\log(\abs{x})
    \end{equation*}
    where $\hat u, v$ are analytic functions ($v$ is null if the dimension $n$ is odd).
\end{theorem}
Using the notation of \cref{eq:ansatz}, Kodaira's result is equivalent to say that the homogeneous functions $u_\ell$ are polynomials, provided the operator $L$ is the Laplace-Beltrami operator of some Riemannian manifold.

The main results of the present paper unify and improve John's result and Kodaira's result.
We improve \cref{thm:john} by giving a sharper (and generally optimal) description of the expansion of $u$. Namely, we show (in \cref{thm:main}) that the terms $\hat u_\ell$ appearing in the statement of \cref{thm:john} are homogeneous polynomials of degree $3\ell$ (up to a linear change of coordinates), instead of general analytic functions. 
Furthermore, we identify (in \cref{thm:precisedenominators}) the elliptic operators $L$ such that the fundamental solution has an even simpler expression and as a byproduct we recover \cref{thm:kodaira}.

Our construction of the fundamental solution is \emph{simple}, \emph{explicit} and \emph{canonical} (see \cref{rmk:algebraic,sec:uniqueness} and the outline of the proof given at the end of this introduction).

\begin{theorem}\label{thm:main}
    For $n\geq 2$, consider the elliptic differential operator
    \begin{equation*}
            L \defeq \sum_{1\le i,j\le n} A_{ij}\partial_{ij} + \sum_{1\le i\le n} b_i \partial_i + c,
    \end{equation*}
    where $A_{ij},b_i,c$ are analytic in the unit ball $B_1$ and $A(x)$ is a positive symmetric matrix for every $x\in B_1$. Let $Q\defeq \sqrt{A(0_{\R^n})}$ (i.e., $Q$ is a positive symmetric matrix such that $Q^2 = A(0_{\R^n})$).
    There is a radius $r_0=r_0(n, A, b, c)>0$ such that the following statements hold.
    
    There is an analytic function $v:Q^{-1}(B_{r_0})\to\R$ (which is null if the dimension $n$ is odd) and a sequence of homogeneous functions $(u_\ell)_{\ell\ge 0}:B_{r_0}\to\R$, $u_\ell$ being $\ell$-homogeneous, such that
    \begin{equation}\label{eq:main_expansion}
        u(Qx) \defeq v(x)\log(\abs{x}) + \frac1{\abs{x}^{n-2}}\sum_{\ell\ge 0} u_\ell
    \end{equation}
    converges absolutely in the $C^\infty$ topology in $Q^{-1}(B_{r_0})$ and $u$ satisfies, in the distributional sense, $Lu=\delta$ in $B_{r_0}$. 
    Moreover, for each $\ell\ge0$, we have
    \begin{equation}\label{eq:formulaforuell}
        u_\ell(x) = \frac{p_\ell(x)}{\abs{x}^{2\ell}},
    \end{equation}
    where $p_\ell$ is a $(3\ell)$-homogeneous polynomial.
\end{theorem}

\begin{remark}\label{rmk:algebraic}
    The polynomials $p_\ell$ and the analytic function $v$ that appear in the statement of \cref{thm:main} depend continuously on the coefficients of the Taylor expansions of $A,b,c$ at the origin. More precisely, the coefficients of the polynomials $p_\ell$ and of the $\ell$-homogeneous part of $v$ are explicit polynomial functions of the entries of $Q$ and of the first $O(\ell)$ coefficients of the Taylor expansions of $A, b, c$.
    
    In order to double-check our formulas and to showcase how explicit our construction is, we implemented a \texttt{C++} program \cite{ComputingFundSol} that, given in input the operator $L$, produces the fundamental solution following our construction.
\end{remark}

In the next Theorem we give a precise description of the simplifications which occur in \cref{eq:formulaforuell} if $p_\ell$ is divisible by $\abs{x}^2$.
We discover that the growth of the complexity of the homogeneous terms $u_\ell$ (which is encoded by the degree of the polynomials $p_\ell$ after the factors $\abs{x}^2$ are simplified) is determined by the Taylor expansion of the principal symbol of $\lapl-L$ evaluated on $x=\xi$ (that is the function $\alpha$ in the statement below).
For clarity we decided to state this result only for operators with $A(0_{\R^n})=\id$; this assumption can always be obtained with a linear change of coordinates.
\begin{theorem}\label{thm:precisedenominators}
    Under the same assumptions of \cref{thm:main}, with the additional assumption $A(0_{\R^n})=\id$; let us consider the analytic function 
    \begin{equation*}
        \alpha(x)\defeq \sum_{i,j}(\delta_{ij}-A_{ij}(x))x_ix_j
    \end{equation*}
    and let $\lambda\ge3$ be the smallest natural number such that the $\lambda$-homogeneous part of $\alpha$ is not divisible by $\abs{x}^2$ ($\lambda=\infty$ if $\alpha$ is divisible by $\abs{x}^2$).
    Then, with reference to the statement of \cref{thm:main}, for each $\ell\ge0$, there is a homogeneous polynomial $\tilde p_\ell$ such that $p_\ell = \abs{x}^{2\ell-\lfloor\frac{2\ell}{\lambda-2}\rfloor}\tilde p_\ell$, so that $u_\ell$ can be written as
    \begin{equation*}
        u_\ell(x) = \frac{\tilde p_\ell(x)}{\abs{x}^{\lfloor\frac{2\ell}{\lambda-2}\rfloor}}.
    \end{equation*}
    Moreover, when $\lambda<\infty$, if $\ell$ is divisible by $\lambda-2$ then  $\tilde p_\ell$ is not divisible by $\abs{x}^2$ and therefore no further simplification between the numerator $\tilde p_\ell(x)$ and the denominator $\abs{x}^{\lfloor\frac{2\ell}{\lambda-2}\rfloor}$ is possible.
\end{theorem}
\begin{remark}
    When the function $\alpha$ is divisible by $\abs{x}^2$, \cref{thm:precisedenominators} guarantees that the fundamental solution can be written as
    \begin{equation*}
        u(x) = p(x)\abs{x}^{-(n-2)} + v(x)\log(\abs{x}) ,
    \end{equation*}
    where $p,v$ are analytic functions and $v\equiv 0$ if the dimension $n$ is odd. This recovers \cite[Theorems 2.6, 3.3]{KhenissyRebaiYe2010} (in the paper it is considered only the case $A\equiv\id$). Moreover, it provides an alternative proof of \Cref{thm:kodaira}. Indeed, if $(x_1,x_2,\dots, x_n):U\subseteq M\to\R^n$ is a normal chart for an analytic Riemannian manifold $(M, g)$; the Laplace-Beltrami operator of a smooth function $\varphi\in C^{\infty}(M)$ can be written as (using the Einstein notation for repeated indexes)
    \begin{equation*}
        \lapl_g\varphi = \div(\nabla \varphi) = \frac{1}{\sqrt{\det g}} \partial_i\big(\sqrt{\det g} g^{ij}\partial_{j}\varphi\big) = g^{ij}\partial_{ij}\varphi + \text{lower order terms} ,
    \end{equation*}
    and the matrix $g^{ij}$ satisfies (since we are using a normal chart) $g^{ij}(x)x_ix_j = \abs{x}^2$ which is equivalent to $\alpha\equiv 0$ (for the operator $L$ given by the Laplace-Beltrami operator in this chart).
    
    Finally, also \cite[Theorem 3.1]{Lukasiewicz09} follows from \cref{thm:precisedenominators} exactly as Kodaira's result. Indeed \cite[Theorem 3.1]{Lukasiewicz09} is equivalent to Kodaira up to the presence of lower-order terms in the operators, but for \cref{thm:precisedenominators} the coefficients $b_i$ and $c$ in the operator $L=A_{ij}\partial_{ij}+b_i\partial_i+c$ do not play any role (the value of $\lambda$ depends only on the matrix $A_{ij}$). 
\end{remark}
\begin{remark}
    As a consequence of \cref{thm:precisedenominators}, we see that the denominator $\abs{x}^{2\ell+n-2}$ for the $(\ell-(n-2))$-homogeneous part of the fundamental solution is sharp in general (or equivalently that the degree $3\ell$ for $p_\ell$ cannot be reduced due to unexpected simplifications). 
    Indeed if the operator $L$ is such that $\lambda=3$ then, for any $\ell\ge0$, $p_\ell$ is not divisible by $\abs{x}^2$. Let us emphasize that this is the behavior for a generic choice of the operator $L$ (in the sense of Baire) and, in particular, it is the case for $L\defeq \lapl + x_1\partial_{11}$.
    More precisely, the statement would be false if we were to require that $u_\ell$ was given by a polynomial divided by $\abs{x}^{2\ell+(n-2)-2}$ (even for just one value of $\ell\ge 0$).
\end{remark}

\subsection{Outline of the proof}
Our proof of \cref{thm:main} follows a classical scheme and, at the same time, is essentially different from the approaches used in the past for the problem. 
While the construction of the fundamental solution is very natural (i.e., we employ the standard method to invert the perturbation of an invertible operator), the proof of the convergence requires novel techniques.
In this section we showcase the key points of our work.

For simplicity, let us assume that $L\defeq \lapl + x_1\partial_{11}$ and that the dimension $n$ of the ambient is odd, so that we do not have to deal with logarithms.
Let $T\defeq (\lapl-L)\circ \lapl^{-1} = x_1\partial_{11}\circ \lapl^{-1}$ (notice that our definition of $T$ is ambiguous, since the inverse of the Laplacian is not well-defined).
One can check that, at least formally, it holds
\begin{equation*}
    L\Big(\lapl^{-1}\big(\delta + T\delta + T^2\delta + \cdots + T^{\ell-1}\delta\big)\Big) = \delta - T^{\ell}\delta .
\end{equation*}
Hence, if we were able to prove that $T^{\ell}\delta\to0$ in a strong norm as $\ell\to\infty$, we would have successfully built the fundamental solution
\begin{equation}\label{eq:intro_fund}
    u(x) \defeq \lapl^{-1}\Big(\sum_{\ell\ge 0} T^\ell\delta\Big) .
\end{equation}
Let us devote some time to understand what is the correct $\lapl^{-1}$ to consider.
If one chooses the inverse Laplacian $\lapl_0^{-1}$ induced by null Dirichlet boundary conditions on a sufficiently small sphere $\partial B_{r_0}$, it is not hard to check that the map $T$ is a contraction (in $L^2(B_{r_0})$) and thus the desired convergence holds. On the other hand, we are not able to control sufficiently the algebraic structure the function $T^\ell\delta$ to deduce the fine properties of the fundamental solution we are looking for\footnote{ The main issue is that the homogeneity of $T^\ell\delta$ does not increase as $\ell\to\infty$. This makes it very hard to study the homogeneous parts of the fundamental solution \cref{eq:intro_fund} as required to prove \cref{thm:main}. Furthermore, even in odd dimension, $T^\ell\delta$ may contain an analytic part (which is not admitted in the statement).}.

The expansion we desire to show for the fundamental solution is \emph{asymptotic with respect to homogeneity}, hence we seek a $\lapl^{-1}$ operator which behaves as dimensional analysis suggests, namely $\lapl^{-1}\varphi$ shall be $(\ell+2)$-homogeneous, whenever $\varphi$ is $\ell$-homogeneous. We construct $\lapl^{-1}$ so that, given a $k$-homogeneous polynomial $p$ and an odd integer $k\in\Z$, $\lapl^{-1}(p\abs{x}^h)$ has the form $q\abs{x}^{h-2}$ where $q$ is another $k$-homogeneous polynomial.
With this choice of the inverse Laplacian one can check that, at least on a formal level, the fundamental solution \cref{eq:intro_fund} has the structure described by \cref{thm:main}. 
On the other hand, the convergence of \cref{eq:intro_fund} is unexpected since there are no \emph{analytic corrections} when one uses this $\lapl^{-1}$.
In fact, it requires hard work to prove the convergence $T^\ell\delta\to0$ as $\ell\to\infty$ because the operator $T$ is not bounded, as shown by the following counterexample. Let $p$ be a homogeneous harmonic polynomial of degree $k\gg 1$. If $\nu={x}/{\abs{x}}$, we have the formula (we are using Euler's formula for the radial derivative of homogeneous functions and \cref{lem:laplacianoperator})
\begin{equation*}
    x_1\partial_{\nu\nu} \lapl^{-1}p = x_1\frac{(k+2)(k+1)}{4k+2n}p \approx x_1\frac{k}4 p.
\end{equation*}
Since there is no reason to expect that $\partial_{11}$ behaves much better than $\partial_{\nu\nu}$ and, for any reasonable norm, $\norm{x_1p}$ is comparable to $\norm{p}$, we deduce that the operator $T$ is not bounded as $k\to\infty$.

The core of our work is establishing the convergence of $T^\ell\delta\to0$ as $\ell\to\infty$.
To do so, we exploit a decomposition in harmonic components which is tailored to our needs: for each $k\ge 0$ and $h>-k-n$, we consider the space $\H^{k,h}$ of functions given by $p\abs{x}^h$ where $p$ is a $k$-homogeneous harmonic polynomial and we consider (countable) sums of elements of these spaces.
Through a precise study of the action of the map $T$ on the spaces $\H^{k,h}$ we reduce the convergence problem to estimating the weight of a path on an infinite directed weighted graph whose vertices are parametrized by $\Z^2$ (the vertex $(k,h)\in\Z^2$ represent the $\H^{k,h}$; the weights of the edges correspond to the norm of $T$ between the respective spaces), which is then established through a neat combinatorial argument (see \cref{prop:crux}).

\subsection{On the covariance of the construction}\label{sec:uniqueness}
In odd dimension the fundamental solution we construct is unique in a suitable sense (see \cref{eq:structural_property}) and it is coordinate invariant, while in even dimension these properties do not hold.
These observations require an argument since the construction is not patently coordinate invariant.

For notational convenience, we assume everywhere that $A(0_{\R^n}) = \id$ (so that the matrix $Q$ appearing in \cref{thm:main} is the identity).

Let us first consider the case $n$ odd. 
As a direct consequence of \cref{thm:main}, we observe that the solution to $Lu=\delta$ we construct satisfies
\begin{equation}\label{eq:structural_property}
    u(x) = \abs{x}U(x, \abs{x}^{-2}) ,
\end{equation}
where $U(\xi,q)$ is a power series\footnote{We are skipping some convergence technicalities in order not to hide the point of this section. Namely one should say that the series which defines $\hat U(x, \abs{x}^{-2})$ converges absolutely when ordered by homogeneity (with respect to $\abs{x}$) in a small neighborhood of the origin. Then, when proving the covariance of our construction, one should check the validity of this convergence.} in the $(\xi,q)$ variables containing only terms $\xi^\alpha q^j$ for $\alpha_1+\alpha_2+\cdots+\alpha_n+j\ge -(n-1)$, so that the minimum homogeneity of $u(x)$ is $-(n-2)$. Let us show that this condition uniquely identifies the fundamental solution $u$ we construct.
Take another function $\hat u$ so that $L\hat u =\delta$ and
\begin{equation*}
    \hat u(x) = \abs{x}\hat U(x, \abs{x}^{-2}) ,
\end{equation*}
where $\hat U$ is a suitable power series. Then, the difference of the two functions satisfies $L(u-\hat u) = 0$ and therefore, by elliptic regularity, it is analytic. If, by contradiction, $u\not=\hat u$, then $\abs{x}(U-\hat U)(x, \abs{x}^{-2})$ would be analytic and nonzero, which is impossible since $\abs{x}$ is not analytic. More precisely, if $u\not=\hat u$ then there exists $\ell\in\Z$ so that $u(x)-\hat u(x) = \abs{x}p_\ell(x)+ \bigo(\abs{x}^{\ell+2})$ where $p_\ell$ is a nonzero $\ell$-homogeneous rational function. Hence, if $u(x)-\hat u(x)$ is smooth then $\abs{x}p_\ell(x)$ is a polynomial, which is impossible and shows the contradiction.

From this uniqueness result, one may also deduce that the fundamental solution we construct behaves nicely under change of coordinates.
Let $\Phi=(\Phi_1, \Phi_2\dots, \Phi_n)\colon\Omega\to\widetilde \Omega$ be an analytic diffeomorphism between two neighborhoods of the origin such that $\Phi(x) = x + \bigo(\abs{x}^2)$.
Let $\tilde L$ be the operator obtained from $L$ by changing the coordinates according to $\Phi$, namely $\tilde L$ is the unique operator such that $(\tilde L\eta)\circ \Phi = L(\eta\circ \Phi)$ for any smooth function $\eta$.
The operator $L$ can be written as $L= A_{ij}(x)\partial_{ij} + b_i(x)\partial_i + c(x)$ and one can check through explicit computations that $\tilde L = \tilde A_{ij}(y)\partial_{ij} + \tilde b_i(y)\partial_i + \tilde c(y)$, where (using the Einstein notation for repeated indices)
\begin{align}
    \tilde A_{ij}(\Phi(x)) &\defeq A_{i'j'}(x)\partial_{i'}\Phi_i(x)\partial_{j'}\Phi_j(x),
    \label{eq:coordinates_changeA}\\
    \tilde b_{i}(\Phi(x)) &\defeq A_{i'j'}(x)\partial_{i'j'}\Phi_i(x) + b_{i'}(x)\partial_{i'}\Phi_i(x), \label{eq:coordinates_changeb}\\
    \tilde c(\Phi(x)) &\defeq c(x) \label{eq:coordinates_changec}.
\end{align}
Notice that $\tilde A, \tilde b, \tilde c$ are analytic functions and $\tilde A(0_{\R^n}) = \id$.
Let $u$ and $\tilde u$ be the fundamental solutions, for the operators $L$ and $\tilde L$ respectively, built according to our construction. We claim that $\tilde u(y) = u(\Phi^{-1}(y))$. Thanks to the above mentioned uniqueness result, since $\tilde L(u\circ\Phi^{-1}) = (Lu)\circ \Phi^{-1} = \delta\circ \Phi^{-1} = \delta$, it is sufficient to show that $u\circ \Phi^{-1}$ can be written as \begin{equation*}
    u\circ \Phi^{-1}(x) = \abs{x}\tilde U(x, \abs{x}^{-2}) ,
\end{equation*}
for some $\tilde U$ which is a power series having the above mentioned properties. One can check, by Taylor expanding $t\mapsto \sqrt{1+t}$, that there exists a power series $F$ so that
\begin{equation*}
    \abs{\Phi^{-1}(x)} = \abs{x}F(x, \abs{x}^{-2}),
\end{equation*}
hence, using \cref{eq:structural_property}, we get
\begin{equation*}
    u\circ \Phi^{-1}(x) = \abs{x}F(x, \abs{x}^{-2})U(\Phi^{-1}(x),  \abs{x}^{-2}F(x, \abs{x}^{-2})^{-1})
\end{equation*}
and this is sufficient to conclude by choosing $\tilde U(\xi, q) \defeq F(\xi, q)U(\Phi^{-1}(\xi), qF(\xi, q)^{-1})$.

Hence, in odd dimension, the fundamental solution we construct is uniquely characterized by its structural properties and is invariant under change of coordinates.

Let us now consider the even dimensions. First of all, if $u$ is the solution we construct, also $u+1$ (and more generally $u+\eta$ where $\eta$ satisfies $L\eta = 0$) can be expressed as in \cref{eq:main_expansion} by replacing $p_{n-2}$ with $p_{n-2} + \abs{x}^{3(n-2)}$ (we need that $n$ is even to have that the additional term is a polynomial).
Still, since we don't make any arbitrary choice in our construction, one may expect that the fundamental solution we obtain is invariant under change of coordinates (that is, if we repeat the construction after changing coordinates, we get the initial fundamental solution composed with the change of coordinates -- as in the odd dimensional case). It turns out that this is not true, as the following counterexample shall clarify.

Let $n=2$, $L=\lapl$ (that is $A_{ij} = \delta_{ij}$, $b_i=0$, $c=0$), our construction produces the standard fundamental solution $u(x)\defeq \log(\abs{x})$ (up to a multiplicative coefficient that we drop). Consider the diffeomorphism $\Phi(x_1, x_2) = (x_1, x_2 + x_1^2)$, whose inverse is given by $\Phi^{-1}(x) = (x_1, x_2 - x_1^2)$. Applying the formulas \cref{eq:coordinates_changeA,eq:coordinates_changeb,eq:coordinates_changec} we obtain that the operator $L$ becomes $\tilde L = \tilde A_{ij}\partial_{ij} + b_i\partial_i + c$ after the change of coordinates, where
\begin{align*}
    \tilde A(x) &= 
    \begin{pmatrix}
        1 & 2x_1\\
        2x_1 & 1+4x_1^2
    \end{pmatrix} ,\\
    \tilde b(x) &= \begin{pmatrix} 0 & 2 \end{pmatrix} ,\\
    \tilde c(x) &= 0 .
\end{align*}
For the operator $\tilde L$, our construction yields the fundamental solution $\tilde u$ which satisfies (as can be checked with \cite{ComputingFundSol})
\begin{equation*}
    \tilde u(x) = \log(\abs{x}) + \frac14 (x_2^3 - 3x_1^2x_2)\abs{x}^{-2} + \bigo(\abs{x}^2\log(\abs{x})).
\end{equation*}
It remains only to check that $\tilde u \not= u\circ\Phi^{-1}$. Let us compute the first terms of the expansion of $u\circ\Phi^{-1}$,
\begin{align*}
    u\circ\Phi^{-1}(x) 
    &= \frac12 \log(x_1^2 + x_2^2 + x_1^4 - 2x_1^2x_2)
    = \log(r) + \frac12 \log\Big(1 + \frac{x_1^4 - 2x_1^2x_2}{x_1^2 + x_2^2}\Big)
    \\
    &= \log(\abs{x}) - x_1^2x_2 \abs{x}^{-2} + \bigo(\abs{x}^2).
\end{align*}
The latter two formulas show that $u\circ\Phi^{-1}\not= \tilde u$, hence our construction is not invariant under change of coordinates in even dimension.

\subsection{Acknowledgements}
Both authors received funding from the European Research Council under the Grant
Agreement No. 721675 ``Regularity and Stability in Partial Differential Equations (RSPDE)''. Federico Franceschini has also been supported by Swiss National Science Foundation Ambizione Grant PZ00P2 180042. 

The authors would like to thank Matteo Bonforte and Alessio Figalli for the useful discussions on the topics of the paper. They are also grateful to Peter Hintz for his valuable comments on a draft of the paper which led to the addition of \cref{sec:uniqueness}.

\section{Harmonic decomposition}\label{sec:harmonicdecomposition}
We will work in $\R^n$ with $n\geq 2$. We denote with $B_r$ the ball of radius $r>0$ centered at the origin $0_{\R^n}$.
The Laplacian operator is $\lapl\defeq \partial_{11} + \partial_{22} + \cdots + \partial_{nn}$. The Japanese bracket of a value $x\in\R$ is $\jap{x}\defeq 1+\abs{x}$ (notice that this is not the usual definition of Japanese bracket). We denote with $\N$ the set of nonnegative integers $\{0,1,2\dots\}$. For $x\in\R$, we denote with $\lfloor x\rfloor$ the largest integer non greater than $x$. The symbol $\delta$ denotes the Dirac delta distribution at $0_{\R^n}$, namely $\langle{\delta,\varphi}\rangle=\varphi(0)$ for any $\varphi\in C^{\infty}_c(\R^n)$.

We will make extensive use of the following spaces of weighted harmonic homogeneous polynomials.
\begin{definition} 
    Consider the following subset of $\Z^2$
    \begin{equation*}
        \Omega\defeq\{(k,h)\in\Z^2 : k\geq 0, k+h>-n\}.
    \end{equation*}
    For each $k\ge 0$, let $\H^k$ be the vector space of $k$-homogeneous polynomials in $n$ variables. For each $k\geq 0, h \in\Z$, we define the following subspace of $C^\infty(\R^n\setminus\{0\})$:
    \begin{equation*}
        \H^{k,h} \defeq \begin{cases}
            \abs{x}^h\H^k &\text{when $h\not\in\{0, 2, 4, \dots\}$,} \\
            \abs{x}^h\H^{k}\oplus \abs{x}^h\log\abs{x}\H^{k} &\text{when $h\in\{0,2,4,\dots\}$.}
        \end{cases}
    \end{equation*}
    For an element $f\in \H^{k,h}$, we define its norm $\norm{f}$ as
    \begin{equation*}
        \norm{f}^2 \defeq \begin{cases}
            \int_{\partial B_1} p^2 &\text{if $h\not\in\{0, 2, 4, \dots\}$ and $f=p\abs{x}^h$,} \\
            \int_{\partial B_1} (p^2+q^2) &\text{if $h\in\{0,2,4,\dots\}$ and $f=p\abs{x}^h+q\abs{x}^h\log(\abs{x})$.}
        \end{cases}
    \end{equation*}
    We denote by $\H^{\bullet,\bullet}\defeq \prod_{(k,h)\in\Omega} \H^{k,h}$ and identify its elements with the formal series $f=\sum_{k,h}f_{k,h}$, where $f_{k,h}\in \H^{k,h}$, in this setting we call the \textit{support of }$f$ the set of the $(k,h)\in\Omega$ for which $f_{k,h}\neq 0$. We have the natural projections:
    \begin{equation*}
        \pi_{k,h}\colon \H^{\bullet,\bullet}\to \H^{k,h}.
    \end{equation*}
\end{definition}
\begin{remark}
    If $(k,h)\in\Omega$ then $\H^{k,h}\subset L^1_{loc}(\R^n)$. More generally, $\H^{k,h}\subset W^{k+h-n-1,1}_{loc}(\R^n)$.
\end{remark}
\begin{remark}\label{rmk:basis}
    When $h\not\in\{0, 2, 4, \dots\}$ we have a linear isomorphism $\H^k\to\H^{k,h}$ given by $p\mapsto\abs{x}^h p$. Similarly, when $h\in\{0, 2, 4, \dots\}$, we have a linear isomorphism $\H^k\times\H^k\to \H^{k,h}$, namely $(p,q)\mapsto p\abs{x}^h+q\abs{x}^h\log(\abs{x})$.
\end{remark}
\begin{definition}\label{def:linearoperator}
    A linear operator on $T\colon \H^{\bullet,\bullet}\to \H^{\bullet,\bullet}$ is, by definition, a collection of linear maps
    \begin{equation*}
        T=\big\{T_{k,h}^{k',h'}\colon \H^{k,h} \to \H^{k',h'}\big\}_{(k,h)\in\Omega}^{(k',h')\in\Omega},
    \end{equation*}
    with the following additional property:
    \begin{equation}\label{eq:propertyofT}
        \{(k,h)\in\Omega :\ T_{k,h}^{k',h'}\neq 0\}\text{ is finite for all $(k',h')\in\Omega$.}
    \end{equation}
    When $f\in\H^{\bullet,\bullet}$ we set 
    \begin{equation*}
        Tf\defeq \sum_{k',h'}\sum_{k,h} T_{k,h}^{k',h'}\big(\pi_{k,h}f\big) ,
    \end{equation*}
    which is well-defined thanks to \cref{eq:propertyofT}.
    
    Given $(k,h), (k', h')\in\Omega$, the operator norm of $T$ between these spaces is defined as
    \begin{equation*}
        \norm{T}_{\mathscr{L}(\H^{k,h}\to\H^{k',h'})} = \sup_{f\in\H^{k,h}:\, \norm{f}=1} \norm {\pi_{k',h'}(Tf)} \,.
    \end{equation*}
\end{definition}
\begin{remark}
    If $S,T\colon \H^{\bullet,\bullet}\to \H^{\bullet,\bullet}$, then $S\circ T$ can be defined by the formula
    \begin{equation*}
        (S\circ T)_{k,h}^{k',h'}\defeq \sum_{i,j} S^{k',h'}_{i,j}\circ T^{i,j}_{k,h},
    \end{equation*}
    which is well-defined thanks to \cref{eq:propertyofT} and enjoys property \cref{eq:propertyofT} itself.
\end{remark}
For all $k\geq 0,h\in \Z$, $p,q\in \H^k$ and $x\neq 0$ we have the algebraic identity
    \begin{equation}\label{eq:laplaciancomputation}
        \lapl(\abs{x}^hp+\abs{x}^h\log\abs{x}q)=h(2k+h+n-2)\abs{x}^{h-2}(p+\log\abs{x}q)+(2h+2k+n-2)\abs{x}^{h-2}q,
    \end{equation}
this formula motivates the choice of the $\H^{k,h}$. More precisely we have the following result concerning the Laplacian.
\begin{lemma}\label{lem:laplacianoperator}
    There exists a linear operator $\lapl\colon \H^{\bullet,\bullet}\to \H^{
    \bullet,\bullet}$ such that for all $(k,h)\in\Omega$ with $k+h>-(n-2)$:
    \begin{enumerate}[ref=(\arabic*)]
        \item for all $f\in \H^{k,h}$,$\lapl f \in \H^{k,h-2}$; \label{it:local3}
        \item using the identifications described in \cref{rmk:basis} ($\H^{k,h}\simeq \H^k$ or $\H^{k,h}\simeq\H^k\times\H^k$ according to $h$), we can write $\lapl$ in matrix form as
    \begin{equation}\label{eq:laplacianmatrix}
        (\lapl)_{k,h}^{k,h-2}=\begin{cases}
             h(2k+h+n-2)& \text{ for }h\not\in\{0, 2, 4, \dots\},\\
            \begin{bmatrix}
        0 & 2k+n-2
        \end{bmatrix} & \text{ for }h=0,\\
            \begin{bmatrix}
        h(2k+h+n-2) & 2(k+h)+n-2 \\
        0 & h(2k+h+n-2)
        \end{bmatrix} & \text{ for }h\in \{2, 4, 6, \dots\}.
        \end{cases}
    \end{equation}\label{it:local4}
        \item for all $f\in \H^{k,h}$, $\lapl f$ then coincides with the distributional Laplacian.\label{it:local5}
    \end{enumerate}

\end{lemma}
\begin{proof}
    For $k+h>-(n-2)$ we have $\H^{k,h}\subset W^{2,1}_{loc}(\R^n)$, so the distributional Laplacian coincides with the pointwise Laplacian which is given by \cref{eq:laplaciancomputation}, hence items \cref{it:local3,it:local5} follow. Item \cref{it:local4} is a restatement of \cref{eq:laplaciancomputation}.
\end{proof}
\begin{remark}\label{rmk:laplacianfullpicture}
    Pick any $f\in\H^{k,h}$. When $2k+h>-(n-2)$ formula \cref{eq:laplaciancomputation} still gives the distributional Laplacian if we interpret it in the principal value sense. On the other hand, when $k,h$ lie on the critical line $\{2k+h+n-2=0\}$, formula \eqref{eq:laplaciancomputation} gives zero while $f$ is not necessarily harmonic. This makes sense since $\lapl f$ is a linear combination of derivatives of the $\delta$ distribution of order $k$ (just considering the homogeneity and the classification of distributions supported at a point). For example $\lapl \H^{0,2-n}\subset \R \delta$, this suggests to define $\H^{0,-n}=\R\delta$.
\end{remark}
Thanks to the previous lemma, we are able to construct a right inverse of $\lapl$, which raises the homogeneity exactly of $2$.
\begin{lemma}\label{lem:inverse_lapl}
There exists an operator $\lapl^{-1}\colon \H^{\bullet,\bullet}\to \H^{
    \bullet,\bullet}$ such that for all $(k,h)\in\Omega$ we have:
    \begin{enumerate}[ref=(\arabic*)]
        \item for all $ f\in\H^{k,h+2}$, $\lapl^{-1} f\in\H^{k,h}$;\label{it:local6}
        \item $\norm{(\lapl^{-1})_{k,h}^{k,h+2}}_{\mathscr{L}(\H^{k,h}\to\H^{k,h+2})}\leq \frac{C(n)}{\jap{h}\jap{\max\{k,\abs{h}\}}}$;\label{it:local7}
        \item $\lapl\circ \lapl^{-1} =\id_{\H^{k,h}}$.\label{it:local8}
    \end{enumerate}
\end{lemma}
\begin{proof}
    Inspecting \cref{eq:laplacianmatrix} we see that $\ker(\lapl)_{k,h}^{k,h+2}$ is nontrivial only for $h=0$, hence the only operators we have some freedom to define (which respect item \cref{it:local6}) are $(\lapl^{-1})_{k,-2}^{k,0}$ for $k\geq 0$. In this case we set
    \begin{equation*}
        (\lapl^{-1})_{k,-2}^{k,0}(p\abs{x}^{-2})\defeq \frac{ p\log\abs{x}}{2k+n-2},\quad \text{ for all }p\in\H^k, k\geq 0.
    \end{equation*}
    This choice gives items \cref{it:local6,it:local8} and also item \cref{it:local7} in the case $h\not=-2$. It remains to prove item \cref{it:local7} when $h+2\not=0$. A short computation using \cref{eq:laplacianmatrix} gives 
    \begin{equation*}
        \norm{(\lapl^{-1})_{k,h}^{k,h+2}}\leq \frac{C}{\abs{(h+2)(2k+h+n)}},
    \end{equation*}
    for some universal number $C>0$. It can then be checked that there is a constant $C(n)$ such that
    \begin{equation*}
        \frac{\jap{h}\jap{\max\{k,\abs{h}\}}}{\abs{(h+2)(2k+h+n)}}\leq C(n),\text{ provided }(k,h)\in\Omega,\, h\neq -2,
    \end{equation*}
    completing the proof of item \cref{it:local7}.
\end{proof}
\begin{remark}\label{rmk:fundamentalsolution}
    Coherently with \cref{rmk:laplacianfullpicture}, we write $\lapl^{-1}\delta=c_2\log\abs{x}$ if $n=2$, and $\lapl^{-1}\delta=c_n\abs{x}^{2-n}$ if $n\geq 3$. Where the constants $c_n\neq 0$ ensures that \cref{it:local10} holds. Since the $\delta$ distribution is $-n$-homogeneous, once again this choice raises the homogeneity by $2$. 
\end{remark}
Before going on we recall the well-known (\cite[Lemma 2.2.1]{ArmitageGardiner2001}) orthogonality between harmonic polynomial with different homogeneities.

\begin{lemma}\label{lem:orthogonality}
For all $k\neq k'$ we have $\H^{k}\perp\H^{k'}$ in $L^2(\partial B_1)$ and in $L^2(B_1)$.
\end{lemma}
\begin{proof}
Let $f\in\H^k$ and $g\in\H^{k'}$.
Thanks to the Gauss-Green identity and the Euler formula for homogeneous functions, we have
\begin{equation*}
(k-k')\int_{\partial B_1}f g=\int_{\partial B_1} \partial_\nu f g-\partial_\nu g f =\int_{B_1}\lapl f g -\lapl g f = 0.
\qedhere
\end{equation*}
\end{proof}
\begin{lemma}[Multiplication by $x_\alpha$]\label{lem:multiplicationoperator}
    For each $\alpha\in\{1,\ldots,n\}$ there exists a linear operator $m_{x_\alpha}\colon \H^{\bullet,\bullet}\to \H^{
    \bullet,\bullet}$ such that, for all $(k,h)\in\Omega$,
    \begin{enumerate}[ref=(\arabic*)]
        \item for all $f\in \H^{k,h}$, $m_{x_\alpha} f\in \H^{k+1,h}\oplus\H^{k-1,h+2}$ (with the convention $\H^{-1,h+2}=\{0\}$);\label{it:local9}
        \item $\norm{(m_{x_\alpha})_{k,h}^{k',h'}}_{\mathscr{L}(\H^{k,h}\to\H^{k',h'})}\leq 1$ for all $(k',h')\in \Omega$;\label{it:local10}
        \item for all $f\in \H^{k,h}$ we have $x_\alpha f = m_{x_\alpha} f$ as $L^1_{loc}$ functions.\label{it:local11}
    \end{enumerate}
\end{lemma}
\begin{proof}
    We first check the case $h=0$. We claim that if $p\in\H^k$, then $x_\alpha p \in \H^{k+1,0}\oplus\H^{k-1,2}$ which gives items \cref{it:local9,it:local10}. Notice the identity
    $$
    \lapl(x_\alpha p)=2\partial_\alpha p =\frac{1}{2k+n-2}\lapl(\abs{x}^2 \partial_\alpha p),
    $$
    which gives the decomposition
    \begin{align}\label{eq:multiplicationdecomposition}
        x_\alpha p = \underbrace{x_\alpha p -\frac{\partial_\alpha p \abs{x}^2}{2k+n-2}}_{\in\H^{k+1,0}} + \underbrace{\frac{\partial_\alpha p \abs{x}^2}{2k+n-2}}_{\in\H^{k-1,2}}.
    \end{align}
    In order to prove item \cref{it:local11} we use that $\H^{k+1,0}\perp \H^{k-1,2}$ in $L^2(\partial B_1)$ and find
    \begin{equation}\label{eq:multiplicationbound}
    \norm*{x_\alpha p -\frac{\partial_\alpha p \abs{x}^2}{2k+n-2}}_{L^2(\partial B_1)}^2+\norm*{\frac{\partial_\alpha p \abs{x}^2}{2k+n-2}}_{L^2(\partial B_1)}^2=\norm{x_\alpha p }_{L^2(\partial B_1)}^2\leq \norm{p}_{L^2(\partial B_1)}^2.
    \end{equation}
    
    We address the general case $h\neq 0$. Take any $f=p\abs{x}^h+q\abs{x}^h\log\abs{x}\in\H^{k,h}$ and use \cref{eq:multiplicationdecomposition} to split it
    \begin{align*}
        x_\alpha f & = \left(x_\alpha p -\frac{\partial_\alpha p \abs{x}^2}{2k+n-2}\right)\abs{x}^h+\left(x_\alpha q -\frac{\partial_\alpha q \abs{x}^2}{2k+n-2}\right)\abs{x}^h\log\abs{x} & (\in\H^{k+1,h})\\
        & \qquad+ \left(\frac{\partial_\alpha p}{2k+n-2}\right)\abs{x}^{h+2}+\frac{\partial_\alpha q }{2k+n-2}\abs{x}^{h+2}\log\abs{x} & (\in \H^{k-1,h+2}).
    \end{align*}
    This immediately gives that $\pi_{k',h'}(x_\alpha f)=0$ unless $(k',h')\in \{(k+1,h);(k-1,h+2)\}$, so items \cref{it:local9,it:local11} are proved. In order to prove item \cref{it:local10} we exploit \cref{eq:multiplicationbound}; for example
    \begin{align*}
       \norm{\pi_{k-1,h+2}(x_\alpha f)}^2= \norm*{\frac{\partial_\alpha p }{2k+n-2}}_{L^2(\partial B_1)}^2 +\norm*{\frac{\partial_\alpha q }{2k+n-2}}_{L^2(\partial B_1)}^2\leq \norm{p}_{L^2(\partial B_1)}^2+\norm{q}_{L^2(\partial B_1)}^2=\norm{f}^2,
    \end{align*}
    and a very similar computation for $\pi_{k+1,h}(x_\alpha f)$.
\end{proof}
\begin{remark}\label{rmk:multiplicationcritical}
    $m_{x_\alpha}$ can be defined also for $k+h=-n$ using formula \cref{eq:multiplicationdecomposition}. Properties \cref{it:local9,it:local10,it:local11} still hold and make sense.
\end{remark}
Incidentally, this gives the following well-known decomposition (\cite[Corollary 2.1.2 ]{ArmitageGardiner2001}).
\begin{theorem}\label{thm:harmonicdecomposition}
    For each $\ell\geq 0$ let $V_\ell$ be the vector space of $\ell$-homogeneous polynomials in $n$ variables, then
    \begin{equation*}
        V_\ell=\H^{\ell,0}\oplus \H^{\ell-2,2}\oplus\ldots\oplus\H^{\ell-2\lfloor\ell/2\rfloor,2\lfloor\ell/2\rfloor},
    \end{equation*}
    and this decomposition is orthogonal in the $L^2(\partial B_1)$ scalar product. Thus we can define, for all $i,j\in \N$, the orthogonal projections $\Pi_{i,j}\colon V_{i+j}\to\H^{i,j}$.
\end{theorem}
\begin{proof}[Sketch of the proof.]
    Uniqueness is a consequence of \cref{lem:orthogonality}. Existence can be proven by induction on $\ell$ iterating \cref{lem:multiplicationoperator}.
\end{proof}
We can now generalize \cref{lem:multiplicationoperator} to any polynomial and, as we will see later, to analytic functions.
\begin{lemma}\label{lem:mul_anal}
     Let $a=\sum_{\abs{\beta}\leq d}c_\beta x^\beta$ be a polynomial of degree $d\geq0$, where $\beta\in\N^n$ is a multi-index and $\{c_\beta\}$ are numbers. Then the operator $m_a\colon \H^{\bullet,\bullet}\to \H^{\bullet,\bullet}$ defined by
     \begin{equation*}
        m_a\defeq \sum_\beta c_\beta \underbrace{m_{x_{1}}\circ\ldots\circ m_{x_{1}}}_{\beta_1 \text{ times }}\circ \ldots \circ \underbrace{m_{x_{n}}\circ \ldots\circ m_{x_{n}}}_{\beta_n \text{ times }},
     \end{equation*}
     enjoys the following properties for all $(k,h)\in\Omega$:
     \begin{enumerate}[ref=(\arabic*)]
         \item\label{it:local12} $(m_a)_{k,h}^{k',h'}=0$ whenever $(k'-k,h'-h)\not\in \{c_1(1,0)+c_2(-1,2): c_1,c_2\in \N\}$;
         \item\label{it:local13} $\norm{m_a}_{\mathscr L(\H^{k,h}\to \H^{k',h'})}\leq \norm{a_{k'+h'-k-h}}_{L^\infty(\partial B_1)}$ for all $(k',h')\in\Omega$. Here $a_\ell=\sum_{\abs{\beta}=\ell}c_\beta x^\beta$ for $\ell\geq 0$, and $a_\ell\equiv 0$ for $\ell<0$;
         \item\label{it:local17} for all $f\in \H^{k,h}$, $a f =m_a f$ as $L^1_{loc}$ functions.
    \end{enumerate}
\end{lemma}
\begin{proof}
    Item \cref{it:local17} holds by construction of $m_a$ and \cref{lem:multiplicationoperator}. For all $(k,h)\in\Omega$ and $p,q\in\H^k$ we have that
    \begin{align*}
        (m_a)_{k,h}^{k',h'}(p\abs{x}^h+q\abs{x}^h\log\abs{x})\defeq \Pi_{k',h'-h}(a_{k'+h'-k-h}\cdot p)\abs{x}^h + \Pi_{k',h'-h}(a_{k'+h'-k-h}\cdot q)\abs{x}^h\log\abs{x},
    \end{align*}
    where the projections $\Pi_{k',h'}$ were defined in \cref{thm:harmonicdecomposition}. This formula gives immediately item \cref{it:local12}. In order to prove item \cref{it:local13} we exploit the fact that $\Pi_{k',h'}$ are $L^2(\partial B_1)$-projectors:
    \begin{align}\label{eq:multiplicationdefinition}
        \norm{\Pi_{k',h'-h}(a_{k'+h'-k-h}\cdot p)}_{L^2(\partial B_1)}^2 & +\norm{\Pi_{k',h'-h}(a_{k'+h'-k-h}\cdot q)}_{L^2(\partial B_1)}^2 \nonumber \\
        &\leq \norm{a_{k'+h'-k-h}}_{L^\infty(\partial B_1)}^2(\norm{p}_{L^2(\partial B_1)}^2+\norm{q}_{L^2(\partial B_1)}^2).
    \end{align}
\end{proof}
\begin{remark}\label{rmk:multipicantionpowerseries}
    Notice that the algebraic construction of $m_a\colon \H\to\H$ can be carried out even when $d=\infty$ (i.e., $a$ is a formal power series), defining $(m_a)_{k,h}^{k',h'}$ by \cref{eq:multiplicationdefinition}. While properties \cref{it:local12,it:local13} are kept also in this generality (they do not depend on $d$), property \cref{it:local17} holds only under further decay assumptions on the sequence  $\{\norm{a_\ell}_{L^\infty(\partial B_1)}\}_{\ell\geq 0}$ (see item \cref{it:multiplicationanalytic} in \cref{prop:convergenceaspects}).
\end{remark}
\begin{lemma}[Differentiation by $\partial_\alpha$]\label{lem:derivative}
    For each $\alpha\in\{1,\ldots,n\}$ there exist a linear operator $\partial_\alpha\colon \H^{\bullet,\bullet}\to \H^{\bullet,\bullet}$ such that for all $(k,h)\in\Omega$ we have
    \begin{enumerate}[ref=(\arabic*)]
        \item\label{it:local14} $\partial_\alpha f\in \H^{k-1,h}\oplus\H^{k+1,h-2}$ for $f\in\H^{k,h}$ (with the convention $\H^{-1,h}=\{0\}$);
        \item\label{it:local15}$\norm{\partial_\alpha}_{\mathscr{L}(\H^{k,h}\to\H^{k-1,h})}\leq C(n)\jap{\max\{k,\abs{h}\}}$ and $\norm{\partial_\alpha}_{\mathscr{L}(\H^{k,h}\to\H^{k+1,h-2})}\leq C(n)\jap{h}$;
        \item\label{it:local16} for any $f\in\H^{k,h}$ with $k+h>-(n-1)$, $\partial_\alpha f$ coincides with the distributional derivative.
        \item\label{it:local19} for any $f\in\H^{k,h}$ with $k+h=-(n-1)$, for any $\beta\in\{1,\ldots,n\}$ we have that $x_\beta\partial_\alpha f =m_{x_\beta}\partial_\alpha f$ as $L^1_{loc}$ functions, where $m_{x_\beta}\partial_\alpha f\in \H^{k-2,h+2}\oplus\H^{k,h}\oplus\H^{k+2,h-2}$ is defined thanks to \cref{rmk:multiplicationcritical}.
    \end{enumerate}
\end{lemma}
\begin{proof}
    For all $p,q\in\H^k$ and $(k,h)\in\Omega$, we have the following pointwise identity in $\R^n\setminus\{0\}$
    \begin{align*}
        \partial_\alpha(p\abs{x}^h+q\abs{x}^h\log\abs{x}) = \underbrace{\partial_\alpha p\abs{x}^h+\partial_\alpha q\abs{x}^h\log\abs{x}}_{\in \H^{k-1,h}} + \underbrace{m_{x_\alpha}((hp+q)\abs{x}^{h-2}+hq\abs{x}^{h-2}\log\abs{x})}_{\in \H^{k+1,h-2}\oplus \H^{k-1,h}} .
    \end{align*}
    We use this formula to define $\partial_\alpha$, hence \cref{it:local14} holds. Since $\H^{k,h}\subset W^{1,1}_{loc}(\R^n)$ when $k+h>-(n-1)$, this formula also holds in the distributions sense in the full $\R^n$, so item \cref{it:local16} holds. The bound on the norm now follows from \cref{lem:multiplicationoperator} and the fact that
    \begin{equation*}
        \norm{\nabla p}_{L^2(\partial B_1)} \leq C(n)\norm{p}_{L^2(\partial B_1)}\text{ for all }p\in\H^k,
    \end{equation*}
    which is a consequence of the identity
    \begin{equation*}
        \frac{1}{n+2(k-1)}\int_{\partial B_1}\abs{\nabla p}^2=\int_{B_1}\abs{\nabla p}^2 = \int_{\partial B_1} (x\cdot \nabla p) p =k\int_{\partial B_1} p^2,\quad \text{ for all } p\in\H^k.
    \end{equation*}
    Finally, let us check \cref{it:local19}. Since for all $x\neq 0$ we have $x_\beta\partial_\alpha f(x) =m_{x_\beta}\partial_\alpha f(x)$ by construction, it suffices to show that the distribution $x_\beta\partial_\alpha f$ is regular (i.e., it is represented by an $L^1_{loc}$ function). This is readily checked, for all $\varphi\in C^\infty_c(\R^n)$ and $q\in\H^{k}$ with $k+h=-(n-1)$ we have
    \begin{align*}
        \left\langle x_\beta\partial_\alpha\left(p \abs{x}^{h}\right),\varphi\right\rangle&=-\lim_{\eps\downarrow 0}\int_{\{\abs{x}>\eps\}} \partial_\alpha(x_\beta\varphi)\abs{x}^h\, dx\\
        &=\lim_{\eps\downarrow 0}\int_{\{\abs{x}>\eps\}}\varphi\underbrace{ m_{x_\beta}\partial_\alpha\left(p\abs{x}^h\right)}_{\in L^1_{loc}(\R^n)}-\lim_{\eps\downarrow 0}\underbrace{\int_{\partial B_\eps}x_\alpha x_\beta p \abs{x}^{h-1}\varphi\, d\sigma}_{=O(\eps^2\eps^{k+h-1}\eps^{n-1})=O(\eps)}\\
        &=\int_{\R^n}m_{x_\beta}\partial_\alpha\left(p\abs{x}^h\right)\varphi\,dx,
    \end{align*}
    and the same computation works for $q\abs{x}^h\log\abs{x}$, therefore 
    $$\left\langle x_\beta\partial_\alpha\big(p\abs{x}^h +q \abs{x}^{h}\log\abs{x}\big),\varphi\right\rangle = \left\langle m_{x_\beta}\partial_\alpha f,\varphi\right\rangle.$$
\end{proof}
What we have done so far is mostly algebraic, let us consider the convergence aspects.
\begin{lemma}
    Let $f\in\H^{k,h}$ for $(k,h)\in\Omega$ such that
    \begin{equation}\label{eq:supportbound}
        \frac{\jap{k+\abs{h}}}{\jap{k+h}}\leq K,
    \end{equation}
    for some number $K>0$. Then for all $\ell \geq 0$ with $k+h>\ell-n/2$ and $r\in(0,1)$ we have:
    \begin{equation}\label{eq:sobolevconvergence}
        \norm{\nabla^\ell f}_{L^{2}(B_r)}\leq C r^{k+h-\ell+\frac n 2}(1+\abs{\log r})\jap{k+h}^\ell\norm{f},
    \end{equation}
    for some $C=C(n,\ell,K)$.
\end{lemma}
\begin{proof}
   For all $p,q\in\H^k$ and $0<r<1$, we have the following inequality
    \begin{align*}\label{eq:norm_on_B_r}
        \norm{p(x)\abs{x}^h}_{L^2(B_r)}&+\norm{q(x)\abs{x}^h\log\abs{x}}_{L^2(B_r)} \\ 
        & \leq 2 r^{\frac n 2+k+h}(1+\abs{\log r})\left(\int_0^1 s^{n+2k+2h}(1+\log^2(s))\frac{ds}{s}\right)^{\frac12}\left(\norm{p}+\norm{q}\right) \\
        & \leq C r^{\frac{n}{2}+k+h}(1+\abs{\log r})(\norm{p}_{L^2(\partial B_1)}+\norm{q}_{L^2(\partial B_1)}),
    \end{align*}
    where $C=C(n)$, as long as $k+h>-n/2$. Let us now estimate any derivative of order $\ell\leq k+h+n/2$:
    \begin{align*}
        \norm{\partial_{\alpha_1\ldots\alpha_\ell} f_{k,h}}_{L^2(B_r)} & \leq \sum_{k',h'} \norm{(\partial_{\alpha_1}\circ\ldots\circ \partial_{\alpha_\ell})^{k',h'}_{k,h}f_{k,h}}_{L^2(B_r)}\\
        & \leq (\ell+1) \max_{k'+h'=k+h-\ell} C(n)r^{\frac{n}{2}+k'+h'}(1+\abs{\log r})\norm{(\partial_{\alpha_1}\circ\ldots\circ \partial_{\alpha_\ell})^{k',h'}_{k,h}f_{k,h}}\\
        & \leq C(n,\ell) r^{\frac{n}{2}+k+h-\ell}(1+\abs{\log r})\jap{\max\{k,\abs{h}\}}^\ell\norm{f_{k,h}}\\
        & \leq C(n,\ell,K)r^{k+h}\jap{k+h}^\ell\norm{f_{k,h}}
    \end{align*}
    where we used \cref{lem:derivative} in order to ensure that $(\partial_{\alpha_1}\circ\ldots\circ \partial_{\alpha_\ell})^{k',h'}_{k,h}\neq 0$ for at most $\ell+1$ values of $(k',h')$, for which we also necessarily have $k'+h'=k+h-\ell$.
\end{proof}
\begin{proposition}\label{prop:convergenceaspects}
    Take $f\in\H^{\bullet,\bullet}$ and set $f_{k,h}\defeq\pi_{k,h} f$. Assume that
    \begin{equation}\label{eq:summability assumptions}
        \sum_{(k,h)\in\Omega}\norm{f_{k,h}}<\infty\text{ and there exists }K > 0\text{ such that if }f_{k,h}\neq 0\text{ then }\frac{\jap{k+\abs{h}}}{\jap{k+h}}\leq K.
    \end{equation}
    Then the formal sum defining $f$ is absolutely convergent in the $C^\infty(B_1)$ metric, more precisely for all $\ell\geq 0$ and $r\in(0,1)$
    \begin{equation*}
        \sum_{\substack{(k,h)\in\Omega\\ k+h\leq \ell}}f_{k,h}
        \text{ has finitely many terms and } 
        \sum_{\substack{(k,h)\in\Omega\\k+h>\ell}}\norm{f_{k,h}}_{C^\ell(B_r)}<C(n,\ell,K,r)\sum_{\substack{(k,h)\in\Omega\\k+h>\ell}}\norm{f_{k,h}}.
    \end{equation*} 
    Hence, as a function, $f\in L^1_{loc}(B_1)\cap C^\infty(B_1\setminus \{0\})$. 
    Moreover, the operations we described can be carried out term by term:
    \begin{enumerate}[ref=(\arabic*)]
        \item If $f_{i,-i-(n-1)}=0$ for all $i\geq 0$, then as a function $f\in W^{1,1}_{loc}(B_1)$, furthermore $\partial_\alpha f$ (as an element of $\H^{\bullet,\bullet}$) satisfies the assumptions \cref{eq:summability assumptions} and coincides with the distributional derivative of $f$. 
        \item\label{it:multiplicationanalytic} Let $(a_\ell)_{\ell\ge 0}$ be a sequence of homogeneous polynomials so that $a_\ell$ has degree $\ell$ and there are $C_a>0$, $\delta_a\in(0,1)$ such that $\norm{a_\ell}_{L^\infty(\partial B_1)}\leq C_a \delta_a^\ell$ for all $\ell\geq 0$. If $a\defeq\sum_{\ell\geq 0}a_\ell$, then $m_a f$ (which was algebraically defined in \cref{rmk:multipicantionpowerseries}) satisfies \cref{eq:summability assumptions} and we have $a\cdot  f = m_a f$ as functions.
        \item The assumptions \cref{eq:summability assumptions} hold for $u\defeq \lapl^{-1}f\in\H^{\bullet,\bullet}$, moreover $u\in W^{2,1}_{loc}(B_1)$ and $\lapl u = f$ in the distributional sense.
    \end{enumerate}
\end{proposition}
\begin{proof}
    We start remarking that under assumption \cref{eq:summability assumptions} we have:
    \begin{equation*}
        \abs{\{(k,h)\in\Omega: k+h\leq \ell, f_{k,h}\neq 0\}}\leq C(K) \jap{\ell}^2\text{ for all }\ell>-n,
    \end{equation*}
    in particular any partial sum $f_\ell:=\sum_{k+h\leq \ell}f_{k,h}\in C^\infty(B_1\setminus\{0\})\cap L^1_{loc}(B_1)$ is finite. Now combining the Sobolev embedding with \cref{eq:sobolevconvergence}, and using that $\sup_{k+h\geq \ell}r^{k+h}\jap{k+h}^\ell<\infty$ we find for all $p>\ell$ and $r\in(0,1)$:
    \begin{equation*}
        \sum_{\substack{(k,h)\in\Omega\\ k+h=p}}\norm{f_{k,h}}_{C^p(B_r)}
        \leq 
        C(n,\ell,K,r)\sum_{\substack{(k,h)\in\Omega\\ k+h=p}}\norm{f_{k,h}}.
    \end{equation*}
    This proves the existence of a function $\widetilde f\in C^\infty(B_1\setminus\{0\})\cap L^1_{loc}(B_1)$ such that
    \begin{equation*}
        \widetilde f(x) = \sum_{(k,h)\in\Omega}f_{k,h}(x)\text{ for all }x\neq 0.
    \end{equation*}
    
    We remark that we can explicitly recover the $f_{k,h}$ from $\widetilde f$ first doing a blow-up (scaling by $r^\ell$ or $r^\ell\log r$) to recover the different homogeneous parts and then exploiting the orthogonality of the $\H^{k}$ to split the polynomials into their harmonic pieces.
    
    Let us prove just item \cref{it:multiplicationanalytic}, as the others are very similar and simpler. We check first that $m_a f$ satisfies both assumptions in \cref{eq:summability assumptions}. Let us start with the summability, for all fixed $(k,h)\in\Omega$ it holds (the pairs $(k',h')$ are implicitly in $\Omega$)
    \begin{align*}
        \sum_{k',h'}\norm{(m_a)_{k,h}^{k',h'}f_{k,h}} & \leq \norm{f_{k,h}} \sum_{k',h'}\norm{a_{k'+h'-k-h}}_{L^\infty(\partial B_1)}\\
        &\leq C_a \norm{f_{k,h}}\sum_{\ell\geq k+h} \ 
        \sum_{\substack{(k'-k,h'-k)\in \N(1,0)+\N(-1,2)\\k'+h'=\ell}} \delta^{k'+h'-k-h}\\
        &\leq C_a\norm{f_{k,h}}\sum_{\ell'\geq 0}(\ell'+1)\delta^{\ell'}=C_a C(\delta_a)\norm{f_{k,h}},
    \end{align*}
    where we used that, thanks to \cref{lem:mul_anal}, 
    \begin{equation*}
        \{(k',h')\in\Omega:\, k'+h'=\ell\}\cap \big((k,h)+\N(1,0)+\N(-1,2)\big)     
    \end{equation*}
    has at most $\ell-(k+h)+1$ elements. 
    Therefore, summing this inequality over $(k,h)\in\Omega$, we find
    \begin{align*}
        \sum_{(k,h)\in\Omega}\sum_{(k',h')\in\Omega}\norm{(m_a)_{k,h}^{k',h'}f_{k,h}}\leq C(C_a,\delta_a)\sum_{(k,h)\in\Omega}\norm{f_{k,h}}<\infty.
    \end{align*}
    
    We now check that the second condition in \cref{eq:summability assumptions} holds for $m_a f$. We first give an equivalent characterization of this condition: let $E\subset \Omega$ be any set, then
    \begin{equation}\label{eq:supportassumption}
        \sup_{E}\frac{\jap{k+\abs{h}}}{\jap{k+h}} < \infty\Leftrightarrow \text{ exists }c<1\text{ such that } h\geq -ck-n\text{ for all }(k,h)\in E.
    \end{equation}
    Take any $(k,h)\in\Omega$ such that $(m_af)_{k,h}\neq 0$, then by \cref{lem:mul_anal} there must exist $(\bar k,\bar h) \in\Omega$ with $(k-\bar k, h-\bar h)\in \mathcal C$ and $f_{\bar k,\bar h}\neq 0$. In particular by \cref{eq:supportassumption}, for some $c_f<1$, we have $\bar h>-c_f\bar k -n$, so one finds
    \begin{equation*}
        h =(h-\bar h)+\bar h\geq (\underbrace{h-\bar h}_{\geq 0})+\delta_f(k-\bar k)-\delta_f k-n\geq \delta_f(\underbrace{k+h-\bar k-\bar h}_{\geq 0})-\delta_f k-n\geq -\delta_f k-n.
    \end{equation*}
    As the only assumption on $(k,h)$ was that $(m_a f)_{k,h}\neq 0$, this proves that \cref{eq:supportassumption} holds for $m_a f$.
    
    To conclude we observe that by \cref{lem:mul_anal}, for all fixed $N>1$, we have
    $$
    \sum_{0\leq \ell\leq N}a_\ell\sum_{k+h\leq N}f_{k,h}=\sum_{k+h\leq 2N} \pi_{k,h}(m_a f)\text{ as functions,}
    $$
    then letting $N\to\infty$ and using what we already proved so far we find $a\widetilde f=\widetilde{m_a f}$.
\end{proof}

\section{Combinatorial bounds for weighted paths in \texorpdfstring{$\Z^2$}{Z}}
In this section we study the paths of some directed weighted graphs on $\Z^2$ as they will be used to encode the norm of (iterations of) linear operators from $\H^{\bullet,\bullet}$ into $\H^{\bullet,\bullet}$.
Given a point $p\in\Z^2$, we denote with $p_x$ and $p_y$ its coordinates, namely $p=(p_x,p_y)$.

\begin{definition}\label{def:graph}
    A directed weighted $\Z^2$-graph $G$ is, formally, a map $G:\Z^2\times\Z^2\to[0,\infty)$; where $G(p, q)$ represents the weight on the edge from point $p$ to point $q$ ($G(p,q)=0$ is equivalent to the absence of the edge). 
    
    A path of length $\ell$ is a sequence of $\ell+1$ vertices $p_0,p_1,\dots, p_\ell$ such that $G$ contains the edges $p_i\to p_{i+1}$ for all $0\le i <\ell$ (i.e., $G(p_i, p_{i+1})>0$).
    The $G$-weight of the path is the product of the weights of the edges, that is $G(p_0,p_1)G(p_1,p_2)\cdots G(p_{\ell-1},p_\ell)$.
\end{definition}
\begin{definition}
    Let $G_1, G_2$ be two directed weighted $\Z^2$-graphs.
    
    The notation $G_1\le G_2$ shall be understood in a pointwise sense, i.e., $G_1(p, q)\le G_2(p, q)$ for any $(p, q)\in\Z^2\times\Z^2$.
    
    The sum of the two graphs, denoted by $G_1+G_2$, is the directed weighted $\Z^2$-graph such that $(G_1+G_2)(p, q) = G_1(p) + G_2(q)$.
    
    The product of the two graphs, denoted by $G_1\cdot G_2$, is the directed weighted $\Z^2$-graph such that (this sum will be absolutely convergent in all the cases we will consider)
    \begin{equation*}
        (G_1\cdot G_2)(p, q) = \sum_{r\in\Z^2} G_1(p, r)G_2(r, q) .
    \end{equation*}
    The product $\ell$ times of $G$ with itself is denoted with $G^\ell$.
\end{definition}

\begin{lemma}\label{lem:count_path}
    Let $G$ be a directed weighted $\Z^2$-graph.
    Assume that there is a $v_0\in \Z^2$ such that if $G(p,q)>0$ then $(q-p)\cdot v_0 > \abs{q-p}$ (i.e., the directions of the edges are contained in a cone strictly smaller than a half-plane).
    
    There is a constant $C_1=C_1(\abs{v_0})>0$ such that, for any $p,q\in\Z^2$, the number of paths from $p$ to $q$ using edges in $G$ is less than $C_1^{\abs{q-p}}$.
\end{lemma}
\begin{proof}
    Let $A\defeq \{q-p: G(p,q)>0\}$. We know that $a v_0\ge \abs{a}$ for any $a\in A$. Calling $v\defeq q-p$, the number of paths from $p$ to $q$ of length $\ell$ using edges in $G$ is bounded from above by
    \begin{equation*}
         \abs{\{(a_1,a_2,\dots,a_\ell)\in A^\ell:\, a_1+a_2+\cdots+a_\ell=v\}} ,
    \end{equation*}
    hence it is sufficient to control this quantity.
    Thanks to the assumption involving $v_0$, there is an integer injective linear transformation $S:\Z^2\to\Z^2$ such that $SA \subset \N_{>0}\times \N_{>0}$ (where $\N_{>0}\defeq \N\setminus\{0\}$).
    One can check that the transformation $Sv\defeq v + 2(v\cdot v_0)\, (1, 1)$ satisfies the assumptions.
    
    Given $(a_1,a_2,\dots, a_\ell)\in A^\ell$ with $a_1+a_2+\cdots+a_\ell=v$, we have $Sa_1+Sa_2+\cdots+Sa_\ell=Sv$ and since the map is injective we deduce that
    \begin{align*}
        &\abs{\{(a_1,a_2,\dots,a_\ell)\in A^\ell:\, a_1+a_2+\cdots+a_\ell=v\}} \\
        &\quad\le
        \abs{\{(b_1,b_2,\dots,b_\ell)\in (\N_{>0}\times\N_{>0})^\ell:\, b_1+b_2+\cdots+b_\ell=Sv\}} \\
        &\quad\le
        \abs{\{(x_1,x_2,\dots,x_\ell)\in \N_{>0}^\ell:\, x_1+x_2+\cdots+x_\ell=(Sv)_x\}}\times \\
        &\quad\quad\ \abs{\{(y_1,y_2,\dots,y_\ell)\in \N_{>0}^\ell:\, y_1+y_2+\cdots+y_\ell=(Sv)_y\}}
        \\
        &\quad = \binom{(Sv)_x-1}{\ell-1}\binom{(Sv)_y-1}{\ell-1} \le 4^{\abs{Sv}}
        \le \big(4^{\abs{S}}\big)^{\abs{v}}.\qedhere
    \end{align*}
\end{proof}

We introduce the two directed graphs $G_1$ and $G_2$ which are at the center of our investigation. Refer to \cref{fig:G1G2} for a graphical representation of the two graphs.
\begin{definition}\label{def:G1}
    Let $G_1$ be the directed weighted graph on $\Z^2$ such that
    \begin{align*}
        G_1((k, h), (k-2, h+2)) &= \frac{\jap{\max\{k,\abs{h}\}}}{\jap{h}} & \text{for all $k\ge 2$ and $h\in\Z$}, \\
        G_1((k, h), (k, h)) &= 1 & \text{for all $k\ge 0$ and $h\in\Z$} , \\
        G_1((k, h), (k+2, h-2)) &= \frac{\jap{h}}{\jap{\max\{k,\abs{h}\}}} & \text{for all $k\ge 0$ and $h\in\Z$} , 
    \end{align*}
    and these are all the edges (i.e., all other values of $G_1$ are null).
\end{definition}
\begin{definition}\label{def:G2}
    Let $G_2$ be the directed weighted graph on $\Z^2$ such that
    \begin{equation*}
        G_2((k, h), (k', h')) =
        \begin{cases}
            1 & \text{if $k,k'\ge 0$, $(k,h)\not=(k',h')$ and $(k'-k, h'-h)\in \N(1,0)+ \N(-1,2)$}, \\
            0 &\text{otherwise}.
        \end{cases}
    \end{equation*}
\end{definition}
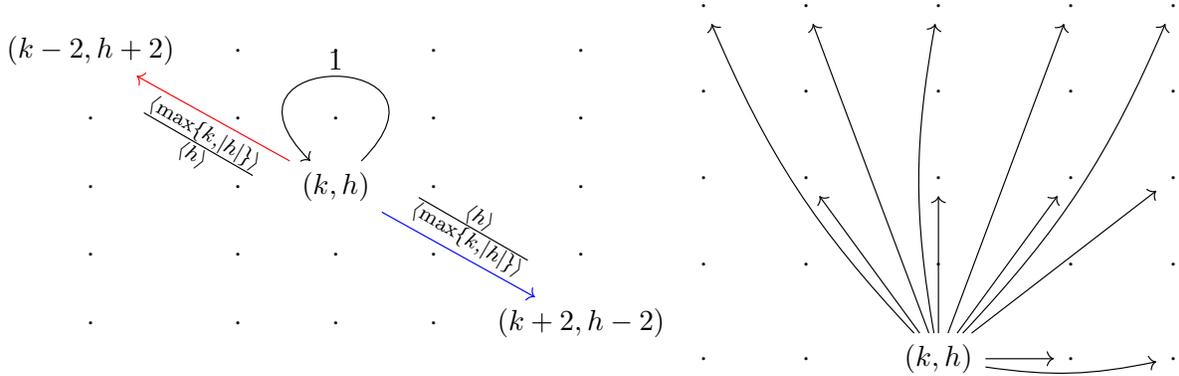
\begin{figure}[htp]
\begin{center}
\begin{tikzcd}[row sep=small, column sep=small]
(k-2, h+2) & \cdot & \cdot & \cdot & \cdot \\
\cdot & \cdot & \cdot & \cdot & \cdot \\
\cdot & \cdot & (k, h) 
\ar[uull, color=red, "{\color{black}\frac{\jap{\max\{k, \abs{h}\}}}{\jap{h}}}" {rotate=-30, anchor=north, font=\normalsize}] 
\ar[ddrr, color=blue, "{\color{black}\frac{\jap{h}}{\jap{\max\{k, \abs{h}\}}}}" {rotate=-30, anchor=south, font=\normalsize}] 
\ar[loop, "1" {above, font=\normalsize}]
& \cdot & \cdot \\
\cdot & \cdot & \cdot & \cdot & \cdot \\
\cdot & \cdot & \cdot & \cdot & (k+2, h-2)
\end{tikzcd}
\begin{tikzcd}
\cdot & \cdot & \cdot & \cdot & \cdot \\
\cdot & \cdot & \cdot & \cdot & \cdot \\
\cdot & \cdot & \cdot & \cdot & \cdot \\
\cdot & \cdot & \cdot & \cdot & \cdot \\
\cdot & \cdot & (k, h) 
\ar[r] \ar[rr, bend right=10] 
\ar[uul] \ar[uu] \ar[uur] \ar[uurr]
\ar[uuuull, bend left=10]
\ar[uuuul] \ar[uuuu, bend left=10]
\ar[uuuur] \ar[uuuurr, bend right=10]
& \cdot & \cdot
\end{tikzcd}
\end{center}
\caption{Diagrams representing $G_1$ and $G_2$.}
\label{fig:G1G2}
\end{figure}

The core of this section is the following proposition, which controls the weight of a path in $G_1+G_2$. Notice that the weight of the edges of $G_1$ are unbounded, so one needs to show some nontrivial cancellation to obtain the result.
\begin{proposition}\label{prop:crux}
    There is a constant $C_2>0$ such that the $(G_1+G_2)$-weight of any path $p_0,p_1,\dots, p_{\ell}$ (the length $\ell$ is arbitrary)  is less than
    \begin{equation*}
        C_2^{\ell+\max_{0\leq i\leq \ell}\abs{p_i}} .
    \end{equation*}
\end{proposition}
\begin{proof}
    Let $e_i\defeq (p_{i-1}, p_i)$ be the edges in the path and let $w_i\defeq (G_1+G_2)(p_{i-1}, p_i)$ be the respective weights. 

    We want to bound $\prod_{0\leq i\leq\ell}w_j$, so we are interested only in those indexes for which $w_i\neq 1$. There are exactly two kind of edges in $G_1+G_2$ with a weight different from $1$.
    Let us say that the edges $(k,h)\to(k-2,h+2)$ are \emph{red} (the only edges which may have a weight larger than $1$) and the edges $(k,h)\to(k+2,h-2)$ are \emph{blue} (the only edges which reduce the $y$ coordinate and the only edges which may have a weight smaller than $1$).
    
    Let us fix a value $h\in\Z$ and consider the set of indexes $S_h\defeq\{i_1<i_2<\cdots<i_{k_h}\}$ defined by the following property: $i\in S_h$ if an only if the edge $e_{i}$ is either red or blue and it crosses the line $\{y=h\}$ (i.e., the two vertices connected by the edge lie on strictly opposite sides with respect to such line).
    Given $i\in S_h$, we have:
    \begin{itemize}
        \item if $e_i$ is a red edge, then $(p_{i-1})_y = h-1$ and $(p_{i})_y = h+1$;
        \item if $e_i$ is a blue edge, then $(p_{i-1})_y = h+1$ and $(p_{i})_y = h-1$. 
    \end{itemize}
    Each edge can cross at most one horizontal line, thus the sets $(S_h)_{h\in\Z}$ are disjoint.
    Since the blue edges are the only ones in $G_1+G_2$ which decrease the $y$-coordinate, $S_h$ cannot contain two consecutive indexes $i_j, i_{j+1}$ which both correspond to red edges. Thus, we have
    \begin{equation}\label{eq:local66}
        \prod_{i\in S_h} w_i \le 
        w_{i_{k_h}} \prod_{1\le j< k_h:\ e_{i_j}\text{ is red}} w_{i_j}w_{i_{j+1}} .
    \end{equation}
    If $e_{i_j}$ is red and $j<i_{k_h}$, then $e_{i_{j+1}}$ must be a blue edge and therefore we have the estimate
    \begin{equation}\label{eq:local67}
        w_{i_j}w_{i_{j+1}} =
        \frac{\jap{\max\{\abs{h-1}, (p_{i_j-1})_x\}}}{\jap{h-1}}
        \frac{\jap{h+1}}{\jap{\max\{\abs{h+1}, (p_{i_{j+1}-1})_x\}}}
        \le 3 \max\Big\{1, \frac{(p_{i_j-1})_x+1}{(p_{i_{j+1}-1})_x+1}\Big\}
    \end{equation}
    Since $(p_i)_x+(p_i)_y$ is a nondecreasing quantity (as $i$ increases), we have
    \begin{equation}\label{eq:local68}
        (p_{i_{j}-1})_x + h-1 \le (p_{i_{j+1}-1})_x + h + 1 \implies (p_{i_{j}-1})_x \le (p_{i_{j+1}-1})_x + 2 
        \implies
        \frac{(p_{i_j-1})_x+1}{(p_{i_{j+1}-1})_x+1}
        \le 
        3
        . 
    \end{equation}
    If $e_{i_k}$ is not red then $w_{i_k}\le 1$, otherwise we have
    \begin{equation}\label{eq:local69}
        w_{i_k} \le \frac{\jap{\max\{\abs{h-1}, (p_{i_k-1})_x\}}}{\jap{h-1}} \le \frac{\max_i \abs{p_i}+1}{\jap{h-1}} .
    \end{equation}
    Joining \cref{eq:local66,eq:local67,eq:local68,eq:local69}, we obtain
    \begin{equation*}
        \prod_{i\in S_h} w_i \le 9^{\abs{S_h}} \frac{\max_i \abs{p_i}+1}{\jap{h-1}} .
    \end{equation*}
    Thanks to the latter inequality, we deduce
    \begin{equation*}
        w_1w_2\cdots w_\ell = \prod_{\abs{h}\le \max_i \abs{p_i}}\prod_{i\in S_h} w_i 
        \le \prod_{\abs{h}\le \max_i \abs{p_i}} C^{\abs{S_h}} \frac{\max_i \abs{p_i}+1}{\jap{h-1}}
        \le 
        C^{\ell} \frac{m^{2m}}{m!^2} \le C^\ell e^{2m},
        .
    \end{equation*}
    where $m=\max_i \abs{p_i} + 1$. The desired statement follows.
\end{proof}

Thanks to \cref{lem:count_path,prop:crux}, we can finally control from above the powers of $G_1\cdot G_2$, which will be crucial in the proof of our main theorem.
\begin{lemma}\label{lem:controlG1G2}
    There is a constant $C_3>0$ such that, for any $\ell\ge 1$, any $p\in \N\times\Z$, and any $q\in\N\times\Z$, we have
    \begin{equation*}
        (G_1\cdot G_2)^\ell(p, q) \le C_3^{\abs{p}+\abs{q}} .
    \end{equation*}
    Moreover, if $(G_1\cdot G_2)^\ell(p, q)\not= 0$ then 
    \begin{equation*}
        q-p\in \Sigma_\ell \defeq
        \{(x,y)\in\Z^2:\, x+y\ge\ell,\, y\ge -2\ell,\, 4x+3y\ge 0,\, 2x+3y\ge 0,\, y\text{ is even}\} .
    \end{equation*}
\end{lemma}
\begin{proof}
    One can check that
    \begin{align*}
        \{q-p:\, (G_1\cdot G_2)(p, q) > 0\} &= 
        \{(x,y)\in\Z^2:\, y\in\{-2,0,2,4,\dots\},\, x+y\ge 1,\, 2x+y\ge -2\} \subseteq
        \Sigma_1 .
    \end{align*}
    If $(G_1\cdot G_2)^\ell(p, q) \not= 0$ then $q-p \in \Sigma_1+\Sigma_1+\cdots + \Sigma_1=\Sigma_\ell$, where the sum is repeated $\ell$ times. Hence, we have shown the second part of the statement.
    From now on we assume $q_x+q_y-(p_x+p_y) \ge \ell$.
    
    Notice that $(G_1\cdot G_2)^\ell(p, q)$ is the sum of the $(G_1\cdot G_2)$-weights of all the paths of length $\ell$ from $p$ to $q$. To bound that quantity we will bound the number of paths and then the maximum weight of a path.
    
    Since the set $\Sigma_1$ is contained in a convex cone strictly smaller than a half-plane, there is a vector $v_0\in\Z^2$ (the choice $v_0=(100,100)$ works) such that $v_0\cdot \sigma \ge \abs{\sigma}$ for any $\sigma\in \Sigma_1$. 
    Hence, the assumption necessary to apply \cref{lem:count_path} on $G_1\cdot G_2$ is satisfied. Therefore the number of paths of length $\ell$ from $p$ to $q$ using edges in $G_1\cdot G_2$ is bounded by $C_1^{\abs{q-p}}$.
    
    Now we prove that any path, with edges in $G_1\cdot G_2$, from $p$ to $q$ is contained in a ball with radius comparable to $\abs{p}+\abs{q}$.
    Let $\Sigma\defeq \{(x,y)\in\R^2:\, 4x+3y\ge 0, 2x+3y\ge 0\}$.
    Since $\Sigma\supseteq \Sigma_1\supseteq \Sigma_2\supseteq\cdots$, any point on a path from $p$ to $q$ which uses only edges of $G_1\cdot G_2$ belongs to 
    \begin{align*}
        (p+\Sigma)\cap&\{(x,y)\in\R^2:\, x+y \le q_x+q_y\}
        =
        p + (q_x+q_y-p_x-p_y)\Big(\Sigma\cap \{(x,y)\in\R^2:\, x+y\le 1\}\Big).
    \end{align*}
    The set $\Sigma\cap \{(x,y)\in\R^2:\, x+y\le 1\}$ coincides with the triangle with vertices $(0,0$), $(3, -2)$, $(-3, 4)$. Since such a triangle is a subset of the ball $B_5$, we deduce that any vertex $r$ on a path from $p$ to $q$ using only edges of $G_1\cdot G_2$ satisfies
    \begin{equation}\label{eq:loc41}
        \abs{r} \le \abs{p} + 5(q_x+q_y-p_x-p_y) \le 11(\abs{p}+\abs{q}) .
    \end{equation}
    
    We proceed by showing that the $G_1\cdot G_2$-weight of a path, with edges in $G_1\cdot G_2$, of length $\ell$ from $p$ to $q$ is bounded by the sum of the $(G_1+G_2)$-weights of at most $3^\ell$ paths, with edges in $G_1+G_2$, of length $2\ell$ from $p$ to $q$.
    Notice that each vertex has at most $3$ outward edges in $G_1$, moreover these edges have length bounded by $2\sqrt 2$.
    Therefore, by definition of product, for each $(r_0, r_2)\in\Z^2\times\Z^2$ there is a set $R_1(r_0, r_2)\subseteq \overline{B(r_0, 2\sqrt{2})}\cap\Z^2$ with cardinality at most $3$ such that
    \begin{equation*}
        (G_1\cdot G_2)(r_0, r_2) 
        =
        \sum_{r_1\in R_1(r_0, r_2)} G_1(r_0, r_1) \cdot G_2(r_1, r_2)
        \le
        \sum_{r_1\in R_1(r_0, r_2)} (G_1+G_2)(r_0, r_1) \cdot (G_1+G_2)(r_1, r_2) .
    \end{equation*}
    Hence, if $p=p_0,p_1,p_2,\dots, p_\ell=q$ is a path from $p$ to $q$ of length $\ell$, we have
    \begin{align*}
        \Big[\text{$(G_1\cdot G_2)$-weight of }& (p_0,p_1,\dots,p_\ell)\Big] 
        = 
        \prod_{i=1}^\ell (G_1\cdot G_2)(p_{i-1},p_i)
        \\
        &\le 
        \prod_{i=1}^\ell \ \sum_{r_i\in R(p_{i-1},p_i)} (G_1 + G_2)(p_{i-1}, r_i) \cdot (G_1+G_2)(r_i, p_i)
        \\
        &=
        \sum_{r_i\in R(p_{i-1},p_i)\ \forall 1\le i\le \ell} \ \prod_{i=1}^\ell (G_1 + G_2)(p_{i-1}, r_i) \cdot (G_1+G_2)(r_i, p_i)
        \\
        &=
        \sum_{r_i\in R(p_{i-1},p_i)\ \forall 1\le i\le \ell} \Big[\text{$(G_1 + G_2)$-weight of $(p_0,r_1,p_1,r_2,p_2,\dots,p_{\ell-1},r_\ell, p_\ell)$}\Big] .
    \end{align*}
    Therefore, the $(G_1\cdot G_2)$-weight of a path of length $\ell$ from $p$ to $q$ is less than the sum of the $(G_1+G_2)$-weights of at most $3^\ell$ path of length $2\ell$ from $p$ to $q$. Furthermore, the mentioned paths in $G_1+G_2$ have all vertices in the set $\{r\in\Z^2: \abs{r} \le 11(\abs{p}+\abs{q})+2\sqrt2\}$ (since they are distant at most $2\sqrt2$ from the closest vertex of the original path and we have \cref{eq:loc41}).
    
    Finally, thanks to \cref{prop:crux}, we know that the $(G_1+G_2)$-weight of any such path of length $2\ell$ is controlled by
    \begin{equation*}
        C_2^{2\ell + 11(\abs{p}+\abs{q})+2\sqrt2} .
    \end{equation*}
    
    Joining the various observations, we have
    \begin{equation*}
        (G_1\cdot G_2)^\ell(p, q) < C_1^{\abs{q-p}}3^{\ell}C_2^{2\ell + 11(\abs{p}+\abs{q})+2\sqrt2} .
    \end{equation*}
    The latter inequality yields the desired result because $\ell\le (q_x+q_y)-(p_x+p_y) \le 2(\abs{p}+\abs{q})$.
\end{proof}

\section{Proof of the main theorems}
We consider an elliptic operator $L=\sum_{i,j}A_{ij}\partial_ij + \sum_i b_i\partial_i + c$ with $(A_{ij})_{i,j}, (b_i)_{i}, c$ analytic functions on $B_1\subseteq\R^n$. We also assume (everywhere but in the proof of \cref{thm:main}, where we handle the general case) that $A_{ij}(0_{\R^n}) = \delta_{ij}$.
Notice that, recalling the theory developed in \cref{sec:harmonicdecomposition}, the operator $L$ can be seen as an operator from $\H^{\bullet,\bullet}$ onto $\H^{\bullet,\bullet}$.
As we have already done in the introduction, let us denote with $T:\H^{\bullet,\bullet}\to\H^{\bullet,\bullet}$ the operator
\begin{equation*}
    T \defeq (\lapl-L)\circ \lapl^{-1} ,
\end{equation*}
where $\lapl^{-1}$ is the inverse Laplacian defined in \cref{lem:inverse_lapl}.

Before going on, let us describe the relationship between the graphs $G_1$ and $G_2$ (see \cref{def:G1,def:G2}) and the map $T$. As explained in the the introduction, the core of the proof of \cref{thm:main} is the understanding of the behavior of the iterations of the map $T$. The graphs $G_1$ and $G_2$ were devised in such a way that 
\begin{equation*}
    \norm{T}_{\mathscr L(\H^{k, h}\to \H^{k', h'})} \le C (G_1\cdot G_2)((k,h), (k',h'))
\end{equation*}
for any $(k, h), (k', h')\in\Omega$, for an appropriate constant $C$. Up to rescaling, thanks to the control we have already obtained on the powers of $G_1\cdot G_2$ in \cref{lem:controlG1G2}, this yields a control on the the iterations of $T$ which is sufficient for our scopes (see \cref{cor:estimateT_highlevel}).

\begin{lemma}\label{lem:estimateT_lowlevel}
    For any $\eps_1>0$ there is $\eps>0$ such that the following statements hold. 
    
    Assume that the analytic functions $(A_{ij})_{i,j}, (b_i)_i, c$ introduced at the beginning of this section satisfy
    \begin{align*}
        &\norm{[x^\ell](\delta_{ij}-A_{ij})}_{L^\infty(B_1)} \le \eps^{\ell} \quad \text{for all $1\le i,j\le n$,}\\
        &\norm{[x^\ell]b_i}_{L^\infty(B_1)} \le \eps^{1+\ell} \quad \text{for all $1\le i\le n$,} \\
        &\norm{[x^\ell]c}_{L^\infty(B_1)} \le \eps^{1+\ell},
    \end{align*}
    for any $\ell\ge 0$ ,where $[x^\ell]g$ denotes the $\ell$-homogeneous part of an analytic function $g$.

    For any $(k,h), (k',h')\in \Omega$ with $k'+h' > k+h$, we have
    \begin{equation}\label{eq:estimateT_generic}
        \norm{T}_{\mathscr L(\H^{k,h}\to\H^{k',h'})}
         \le \eps_1^{(k'+h')-(k+h)} (G_1\cdot G_2)((k,h), (k', h')) .
    \end{equation}
    Furthermore, for any $(k,h)\in \Omega$, we have\footnote{As observed in \cref{rmk:laplacianfullpicture}, morally the $\delta$ distribution belongs to the space $\H^{0, -n}$.}
    \begin{equation}\label{eq:estimateT_delta}
        \norm{\pi_{k,h}(T\delta)} \le \eps_1^{k+h+n} (G_1\cdot G_2)((0,-n), (k, h)) .
    \end{equation}
\end{lemma}
\begin{proof}
    We associate to a linear operator $S:\H^{\bullet,\bullet}\to\H^{\bullet,\bullet}$ a directed weighted $\Z^2$-graph $G_S$ (see \cref{def:graph}) by setting
    \begin{equation*}
        G_S((k,h),(k',h'))\defeq\norm{S}_{\mathscr L(\H^{k,h}\to \H^{k',h'})} .
    \end{equation*}
    Notice that $G_{S_1\circ S_2} \le G_{S_2}\cdot G_{S_1}$.

    By the triangle inequality, we have
    \begin{equation}\label{eq:local18}
        G_T \le \sum_{1\le i,j\le n} G_{\partial_{ij}\circ \lapl^{-1}} \cdot G_{m_{a_{ij}}}
        + \sum_{1\le i\le n} G_{\partial_{i}\circ \lapl^{-1}} \cdot G_{m_{b_{i}}}
        + G_{\lapl^{-1}} \cdot G_{m_{c}} ,
    \end{equation}
    where $a_{ij}\defeq \delta_{ij}-A_{ij}$.
    
    Let $G_+$ be the weighted directed $\Z^2$-graph such that, for each $(k, h)\in \Z^2$, $G((k, h), (k+1, h)) = 1$ and these are all the edges. Let $G_-$ be the weighted directed $\Z^2$-graph such that, for each $(k, h)\in\Z^2$, $G((k, h),(k-1, h))=1$.
    Notice that $G_+\cdot G_- = G_-\cdot G_+$ and it corresponds to a $\Z^2$-graph which has only self-edges with weight $1$, so it is the identity element with respect to the multiplication of $\Z^2$-graphs.
    
    We claim that, when evaluated on the pairs $p, q\in \Omega$, we have (for a universal constant $C=C(n)>0$)
    \begin{align}
        G_{\partial_{ij}\circ\lapl^{-1}} &\le C\cdot G_1 , \label{eq:local19}\\
        G_{\partial_i\circ \lapl^{-1}} &\le C \cdot G_1\cdot G_+ , \label{eq:local20} \\
        G_{\lapl^{-1}} &\le C \cdot G_1 \cdot G_+\cdot G_+ . \label{eq:local21} 
    \end{align}
    We will prove only \cref{eq:local19} since the other two inequalities \cref{eq:local20,eq:local21} can be shown in an analogous way (and their proof is also easier).
    
    Thanks to \cref{lem:derivative,lem:inverse_lapl}, we know that (up to universal multiplicative constants) the graph $G_{\partial_{ij}\circ\lapl^{-1}} = G_{\lapl^{-1}}\cdot G_{\partial_i}\cdot G_{\partial_j}$ is less than or equal to the graph represented by the diagram\footnote{The graph represented by the diagram is a directed weighted $\Z^2$-graph whose edges go from $(k,h)$ to elements of the rightmost layer of the diagram with weights given by the sum of the weights of the paths in the diagram between the corresponding points.}
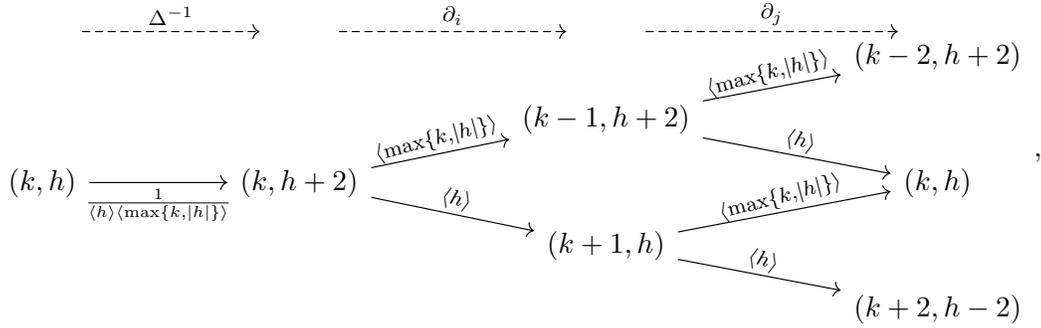
    \begin{center}
    \begin{tikzcd}[row sep=tiny, column sep=huge]
    \phantom{1em} \ar[r,"\lapl^{-1}", dashed, shift right=2.5ex]& \phantom{1em} \ar[r,"\partial_i", dashed, shift right=2.5ex] & \phantom{1em} \ar[r,"\partial_j", dashed, shift right=2.5ex] & \phantom{1em} \\
    & & & (k-2,h+2)\\
    & & (k-1, h+2) \ar[dr, "\jap{h}" dlab] \ar[ur, "\jap{\max\{k, \abs{h}\}}" ulab]& \\
    (k,h)  \ar[r, "\frac{1}{\jap{h}\jap{\max\{k,\abs{h}\}}}" below] & (k,h+2) \ar[dr, "\jap{h}" dlab] \ar[ur, "\jap{\max\{k,\abs{h}\}}" ulab] & & (k,h)\\
    & & (k+1, h) \ar[dr, "\jap{h}" dlab] \ar[ur, "\jap{\max\{k, \abs{h}\}}" ulab] & \\
    & & & (k+2, h-2)
    \end{tikzcd},
    \end{center}
    which is equivalent, up to multiplicative constants, to the graph represented by the diagram (obtained by summing the weights of the paths of length $3$ in the previous diagram)
    
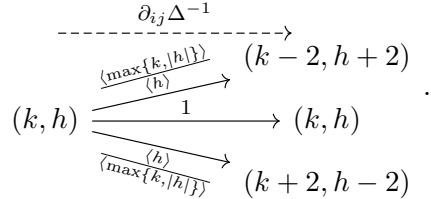
\begin{center}
    \begin{tikzcd}[row sep=tiny, column sep=huge]
    \ar[r,"\partial_{ij}\lapl^{-1}", dashed, shift right=2.5ex] & \phantom{1ex}\\
    & (k-2, h+2) \\
    (k, h) \ar[ur, "\frac{\jap{\max\{k,\abs{h}\}}}{\jap{h}}" ulab] \ar[r, "1"] \ar[dr, "\frac{\jap{h}}{\jap{\max\{k,\abs{h}\}}}" {dlab, below}] & (k, h) \\
    & (k+2, h-2)
    \end{tikzcd}.
    \end{center}
    Since the latter diagram corresponds to $G_1$ (cp. \cref{fig:G1G2}), we have proven \cref{eq:local19}.
    
    We also have the following inequalities between $\Z^2$-graphs, when we evaluate them on the pairs of points $p, q\in \Omega$
    \begin{equation}\label{eq:local22}\begin{aligned}
        G_{m_{a_{ij}}} &\le \eps^{\ell} \cdot G_2 , \\
        G_{m_{b_i}} &\le \eps^{1+\ell} \cdot G_- \cdot G_2 , \\
        G_{m_c} &\le \eps^{1+\ell} \cdot  G_-\cdot G_- \cdot G_2 ,
    \end{aligned}\end{equation}
    where $\ell\defeq q_x+q_y - (p_x+p_y)$.
    These three inequalities follow directly from the definition of $G_2$ (see \cref{def:G2}) and \cref{lem:mul_anal}.
    
    Joining \cref{eq:local18,eq:local19,eq:local20,eq:local21,eq:local22}, we obtain for all pairs of points $p, q\in \Omega$ (we use the Einstein notation for brevity)
    \begin{align*}
        G_T(p, q) &\le  
        \sum_{r\in\Omega:\ r_x+r_y=p_x+p_y} G_{\partial_{ij}\circ \lapl^{-1}}(p, r)  G_{m_{a_{ij}}}(r, q) \\
        &\quad+ \sum_{r\in\Omega:\ r_x+r_y=p_x+p_y+1} G_{\partial_{i}\circ \lapl^{-1}}(p, r) G_{m_{b_{i}}}(r, q) \\
        &\quad+ \sum_{r\in\Omega:\ r_x+r_y=p_x+p_y+2} G_{\lapl^{-1}}(p, r) G_{m_{c}}(r, q) \\
        &\le  
        C\eps^\ell\cdot\sum_{r\in\Omega:\ r_x+r_y=p_x+p_y} G_1(p, r)  G_2(r, q) \\
        &\quad+ C\eps^{1+\max\{0, \ell-1\}}\cdot\sum_{r\in\Omega:\ r_x+r_y=p_x+p_y+1} (G_1\cdot G_+)(p, r) (G_-\cdot G_2)(r, q) \\
        &\quad+ C\eps^{1+\max\{0, \ell-2\}}\cdot\sum_{r\in\Omega:\ r_x+r_y=p_x+p_y+2} (G_1\cdot G_+\cdot G_+)(p, r) (G_-\cdot G_-\cdot G_2)(r, q) \\
        &\le C'\big(\eps^{\ell} + \eps^{1+\max\{0, \ell-2\}}\big)(G_1\cdot G_2)(p, q) \,.
    \end{align*}
    where $\ell\defeq q_x+q_y-(p_x+p_y)$ and $C'=C'(n)>0$ is a universal constant. 
    Since both sides are null whenever $\ell\le 0$, we may assume $\ell\ge 1$. Then, it is elementary to check that $1+\max\{0, \ell-2\} \ge \frac{\ell}2$ and thus we have (for an appropriate universal constant $C>0$)
    \begin{equation*}
        G_T(p, q) \le 2C'\eps^{\frac{\ell}2}\cdot G_1\cdot G_2(p, q) ,
    \end{equation*}
    which implies \cref{eq:estimateT_generic} if we choose $\eps = \eps_1^2/(2C')^2$.
    
    We shift our attention to proving \cref{eq:estimateT_delta}.
    Let $\psi\in\H^{0,-(n-2)}$ be the function such that $\lapl^{-1}\delta=\psi$ considered in \cref{rmk:fundamentalsolution}.
    Thanks to \cref{lem:multiplicationoperator} and \cref{lem:derivative}-\cref{it:local19} we have
    \begin{equation}\label{eq:local112}\begin{aligned}
        &\psi \in\H^{0, -(n-2)} ,\\
        &\partial_i\psi \in \H^{1,-n} 
        \quad\text{for all $1\le i\le n$,}
        \\
        &x_\alpha\partial_{ij}\psi \in \H^{1,-n}\oplus \H^{3,-(n+2)} 
        \quad\text{for all $1\le i,j,\alpha\le n$.}
    \end{aligned}\end{equation}
    Take a function $\varphi\in \H^{k,h}$ and an analytic function $d=\sum_{\ell\ge 0}d_\ell$, where $d_\ell$ is $\ell$-homogeneous, with $\norm{d_\ell}_{L^{\infty}(B_1)}\le \eps^{1+\ell}$. \Cref{lem:mul_anal} tells us that 
    if $(k'+h')-(k+h)\in \N(1,0)+\N(-1, 2)$ then
    \begin{equation*}
        \norm{\pi_{k',h'}(d\cdot \varphi)} \le \norm{\varphi} \eps^{1+(k'+h')-(k+h)}
    \end{equation*}
    and $\pi_{k',h'}(d\cdot\varphi) = 0$ for all other pairs of $(k',h')\in\Z^2$.
    The latter observation, together with \cref{eq:local112} guarantees that $\pi_{k,h}(T\delta)=0$ if $(k,h)\not\in \{(1,-n), (3, -(n+2)\}+\N(1,0)+\N(-1,2)$ and otherwise
    \begin{equation*}
        \norm{\pi_{k,h}(T\delta)} \le C\eps^{1+\max\{k+h+n-2, 0\}}
    \end{equation*}
    for a certain constant $C=C(n)>0$. Exactly as we did for \cref{eq:estimateT_generic}, the desired \cref{eq:estimateT_delta} follows if we set $\eps = \eps_1^2/C^2$.
\end{proof}

\begin{corollary}\label{cor:estimateT_highlevel}
    For any $\eps_2>0$ there is $\eps>0$
    such that, under the same assumptions of \cref{lem:estimateT_lowlevel}, the following statement hold.
    
    Given $\ell\ge 1$, for any $(k,h)\in\Omega\cap\big((0,-n)+\Sigma_\ell\big)$ (where the set $\Sigma_\ell$ is defined in \cref{lem:controlG1G2}) we have
    \begin{equation*}
        \norm{\pi_{k,h}(T^\ell\delta)} \le \eps_2^{k+h+n} ,
    \end{equation*}
    while for all the other $(k,h)\in\Omega$ we have
    \begin{equation*}
        \pi_{k,h}(T^\ell\delta) = 0 .
    \end{equation*}
\end{corollary}
\begin{proof}
    Let $\eps_1\defeq C_3^{-(2n+\sqrt{13})}\eps_2$, where $C_3$ is the constant which appears in \cref{lem:controlG1G2}.
    Thanks to \cref{lem:controlG1G2,lem:estimateT_lowlevel}, if $\eps>0$ is sufficiently small, we have
    \begin{equation}\label{eq:local245}\begin{aligned}
        &\norm{\pi_{k,h}(T^\ell\delta)} 
        \le
        \sum_{(k',h')\in\Z^2} \norm{\pi_{k', h'}(T\delta)}\cdot \norm{T^{\ell-1}}_{\mathscr L(\H^{k',h'}\to\H^{k,h})} 
        \\
        &\quad\le
        \sum_{(k',h')\in\Z^2} \eps_1^{k'+h'+n}(G_1\cdot G_2)((0, -n), (k',h')) \cdot \eps_1^{k+h-(k'+h')} (G_1\cdot G_2)^{\ell-1}((k', h'), (k, h)
        \\
        &\quad=
        \eps_1^{k+h+n} (G_1\cdot G_2)^\ell((0, -n), (k, h)) 
        \le \eps_1^{k+h+n}C_3^{\abs{(k,h)}+n} .
    \end{aligned}\end{equation}
    \Cref{lem:controlG1G2} guarantees that if $(k,h)-(0,-n)\not\in\Sigma_\ell$ then $(G_1\cdot G_2)^\ell((0, -n), (k,h)) = 0$ and thus $\pi_{k,h}(T^\ell\delta)=0$.
    From now on, we assume that $(k,h)\in \Omega\cap\big((0,-n)+\Sigma_\ell)$. 
    It is elementary to check that if $x,y\in\R$ satisfy $x\ge 0$ and $2x+3y\ge 0$, then $\abs{(x,y)}\le \sqrt{13}(x+y)$. Applying such inequality with $x = k$, $y = h+n$, we obtain
    \begin{equation}\label{eq:local246}
        \abs{(k,h)} + n \le \abs{(k, h+n)} + 2n\le \sqrt{13}(k+h+n) + 2n \le (k+h+n)(2n + \sqrt{13}).
    \end{equation}
    Joining \cref{eq:local245,eq:local246} and recalling the definition of $\eps_1$, the statement follows.
\end{proof}

\begin{proof}[Proof of \cref{thm:main}]
    Let us define the operator
    \begin{equation*}
        \tilde L\defeq (Q^{-1}AQ^{-1})_{ij}\partial_{ij} + (Q^{-1}b)_i\partial_i + c,
    \end{equation*}    
    where $Q$ is the matrix considered in the statement such that $A(0_{\R^n})= Q^2$. One can check that, for any function $\varphi$, $\tilde L(\varphi\circ Q)(x) = (L\varphi)(Qx)$. Hence, $u:B_1\to\R$ satisfies $Lu=\delta$ if and only if $\tilde L(u\circ Q) = \frac{\delta}{\det(Q)}$.
    Thus, working with $\tilde L$ instead of $L$, we may assume that $A_{ij}(0_{\R^n})=\delta_{ij}$ (i.e., the operator $L$ behaves as a Laplacian close to the origin).

    Given $r>0$ and a function $g$, let us denote with $g_r$ the function $g_r(x)\defeq g(rx)$. 
    Given $r>0$, let us define the operator
    \begin{equation*}
        L_r\defeq (A_{ij})_r\partial_{ij} + r\cdot(b_i)_r\partial_i + r^2\cdot c_r .    
    \end{equation*}
    One can check that, for any function $\varphi$, one has $L_r \varphi_r\big(\frac xr\big) = r^2 L\varphi(x)$.
    Hence, $u:B_{r_0}\to\R$ satisfies $Lu=\delta$ in $B_{r_0}$ if and only if $L_{r_0}u_{r_0} = r_0^{2-n}\delta$ in $B_1$. 
    Thus, working with $L_{r_0}$ instead of $L$ (with $r_0>0$ chosen appropriately), we may assume (see \cite[Lemma 2.1.10]{KrantzParks2002}) that, for an arbitrarily small fixed $\eps>0$, we have 
    \begin{align*}
        &\norm{[x^\ell](\delta_{ij}-A_{ij})}_{L^\infty(B_1)} \le \eps^{\ell} \quad \text{for all $1\le i,j\le n$,}\\
        &\norm{[x^\ell]b_i}_{L^\infty(B_1)} \le \eps^{1+\ell} \quad \text{for all $1\le i\le n$,} \\
        &\norm{[x^\ell]c}_{L^\infty(B_1)} \le \eps^{1+\ell},
    \end{align*}
    for any $\ell\ge 0$ ,where $[x^\ell]g$ denotes the $\ell$-homogeneous part of an analytic function $g$. 
    
    Let $T\defeq (\lapl-L)\circ \lapl^{-1}:\H^{\bullet,\bullet}\to \H^{\bullet,\bullet}$. Since $L\circ\lapl^{-1}=\id-T$, we have (as an identity between linear operators acting on $\H^{\bullet,\bullet}$)
    \begin{equation}\label{eq:main_formula}
        L\big(\lapl^{-1}(\id+T+T^2+\cdots+T^{\ell-1})(\delta)\big) = \delta-T^\ell\delta . 
    \end{equation}
    Let us choose $\eps>0$ so that \cref{cor:estimateT_highlevel} holds with $\eps_2=\frac12$.
    It follows from \cref{cor:estimateT_highlevel} that, for any $\ell\ge 1$, $T^\ell\delta$ satisfies the assumptions of \cref{prop:convergenceaspects} and $\sum_{\ell\ge 0}\norm{T^\ell\delta} < \infty$.
    Therefore, we deduce that $T^\ell\delta$ is a smooth function outside the origin (in the unit ball), it is $L^1_{loc}$ up to the origin and the series $\sum_{\ell\ge 0}T^\ell\delta$ converges in the $C^{\infty}$-topology, in any compact set of the unit ball, to a function which is smooth outside the origin.
    We obtain also that the formula \cref{eq:main_formula} holds both in the distributional sense and pointwise. Since $T^\ell\delta\to 0$ in the $C^\infty$-topology in any compact set of the unit ball, we deduce
    \begin{equation*}
        L\Big(\lapl^{-1}\big(\sum_{\ell\ge 0} T^\ell\delta\big)\Big) = \delta .
    \end{equation*}
    We have successfully built the fundamental solution, it remains only to check that it has the structure described in the statement. Let $u\defeq \lapl^{-1}\big(\sum_{\ell\ge 0} T^\ell\delta\big)$ be the fundamental solution that we have built.
    Now we want to study the algebraic structure of the homogeneous components of $u$. We want to prove that, for any $\ell\in\Z$, the $(\ell-(n-2))$-homogeneous component of $u$ is given by a homogeneous polynomial of degree $3\ell$ divided by $\abs{x}^{2\ell+(n-2)}$ (plus a polynomial multiplied by $\log(\abs{x})$ if the dimension is even) and it is null if $\ell<0$. 
    Equivalently (both in the case $n$ odd and $n$ even), we want to prove that
    \begin{equation}\label{eq:tmp55}
        k+h \ge -(n-2)
        \quad\text{and}\quad 
        h+n\text{ is even} 
        \quad \text{and} \quad 
        h\ge -(2(k+h+n-2)+(n-2)) .
    \end{equation}
    Recalling the action of $\lapl^{-1}$ (see \cref{lem:inverse_lapl}) and the fact that $\lapl^{-1}\delta\in \H^{0, -(n-2)}$, these two properties would follow from proving that, for any $\ell\ge 1$, if $\pi_{k,h}(T^\ell\delta)\not=0$ then (these formulae are obtained by substituting $h$ with $h+2$ in \cref{eq:tmp55})
    \begin{equation*}
        k+h \ge -n
        \quad\text{and}\quad 
        h+n\text{ is even} 
        \quad \text{and} \quad 
        2k+3(h+n)\ge 0 .
    \end{equation*}
    Since these properties follow from $(k,h)\in\Omega\cap\big((0,-n)+\Sigma_\ell\big)$ which is proven in \cref{cor:estimateT_highlevel}, the proof is concluded.
\end{proof}

It remains to prove \cref{thm:precisedenominators}. We need the following preparatory lemma which is a refinement of \cref{lem:estimateT_lowlevel}.
\begin{lemma}\label{lem:lowest_arrow}
    Let $f\defeq p\abs{x}^h\in \H^{k,h}$, where $(k,h)\in\Omega$. Let $\alpha(x)\defeq \sum_{i,j} (\delta_{ij}-A_{ij}(x))x_ix_j$.
    Then, for any $k'\ge 0$ and $h'<h-2$, it holds $\pi_{k',h'}(Tf)=0$.
    Moreover, if $h<0$, we have, for any $k'\ge 0$, 
    \begin{equation}\label{eq:tmp44}
        \pi_{k',h-2}(Tf) 
        = \frac{h\abs{x}^{h-2}}{(2k+h+n)}\, \Pi_{k', 0}(\alpha\cdot p) ,
    \end{equation}
    where the projection $\Pi_{k',0}$ is defined in \cref{thm:harmonicdecomposition}.
    Analogously, for the $\delta$ distribution we have $\pi_{k',h'}(T\delta)=0$ for any $k'\ge 0$ and $h'<-n-2$. Moreover, for any $k'\ge 0$, we have
    \begin{equation}\label{eq:tmp45}
        \pi_{k', -n-2}(T\delta) = \frac{-n c_n}{\abs{x}^{n+2}} \Pi_{k', 0}(\alpha) ,
    \end{equation}
    where the constant $c_n$ appears in the definition of $\lapl^{-1}\delta$ (see \cref{rmk:fundamentalsolution}).
\end{lemma}
\begin{proof}
If $h'<h-2$, then $\pi_{k',h'}(Tf)=0$ follows from \cref{lem:controlG1G2,lem:estimateT_lowlevel}. For the same reasons, $\pi_{k',h'}(T\delta)=0$ if $h'<-n-2$.
We prove only \cref{eq:tmp44} when $h\not=-2$, the case $h=2$ and \cref{eq:tmp45} can be proven analogously.

First of all, if $f\in\H^{k,h}$, we have $\pi_{k',h'}(b_i\partial_i\lapl^{-1}f + c\lapl^{-1}f)=0$ whenever $h'<h$. Such statement follows from \cref{lem:inverse_lapl,lem:mul_anal,lem:derivative}. Therefore we can assume without loss of generality that $c=0$ and $b_i=0$ for all $1\le i\le n$.

Given an homogeneous harmonic polynomial $p\in\H^k$ and $\alpha\in\{1,2,\dots,n\}$, let $S_\alpha(p)$ be defined as (provided that $2k+n-2>0$)
\begin{equation*}
    S_\alpha(p)\defeq \Pi_{k+1,0}(px_\alpha) = p x_\alpha - \frac{1}{2k+n-2}\partial_\alpha p\abs{x}^2 .
\end{equation*}
Notice that $S_\alpha(p)\in \H^{k+1}$, and if $p\not=0$ then $S_\alpha(p)\not=0$ (because $\abs{x}^2$ cannot divide $p$).
Given $p\abs{x}^h\in \H^{k,h}$, we define $S(p\abs{x}^h)=S(p)\abs{x}^h$.
As a byproduct of the proofs of \cref{lem:multiplicationoperator,lem:derivative}, we have (provided that $2k+n-2>0$) that if $f\in\H^{k,h}$ then
\begin{equation}\label{eq:local100}
    \partial_\alpha f - \frac{h}{\abs{x}^2} S_\alpha(f) \in \H^{k-1, h} .
\end{equation}
Therefore, given $(k,h)\in\Omega$ with $h<0$, $h\not=-2$ and $h+n$ even (this implies $k>0$ or $n>2$),
 the action of $\partial_{ij}\circ\lapl^{-1}$ on $\H^{k,h}$ is represented by the following diagram (whose validity follows from \cref{lem:laplacianoperator})
\begin{center}
    \begin{tikzcd}[row sep=tiny, column sep=small]
    \phantom{1em} \ar[rrr, "\lapl^{-1}", dashed, shift right=2ex]
    &&& 
    \phantom{1em} \ar[r, "\partial_i", dashed, shift right=2ex] 
    & 
    \phantom{1em} \ar[r, "\partial_j", dashed, shift right=2ex] 
    & 
    \phantom{1em} \\
    &&&&(k-1, h+2)& \\
    (k,h) \ar[rrr, "\frac{\abs{x}^2}{(h+2)(2k+h+n)}"]
    &&&
    (k, h+2) \ar[ur, squiggly] \ar[dr, "\frac{h+2}{\abs{x}^2}S_i" below]&&(k, h) \\
    &&&&(k+1, h) \ar[ur, squiggly] \ar[dr, "\frac{h}{\abs{x}^2}S_j" below] &  \\
    &&&&& (k+2, h-2) \\
    &&&&&
    \end{tikzcd},
\end{center}
from which we deduce that, if $f\defeq p\abs{x}^h\in\H^{k,h}$, $\pi_{k', h'}(\partial_{ij}\lapl^{-1}f)\not=0$ implies either $h' \ge h$ or $h'=h-2$ and $k'=k+2$. Furthermore we obtain also
\begin{equation}\label{eq:loc43}
    \pi_{k+2,h-2}(\partial_{ij}\lapl^{-1}f) =
    \frac{h\abs{x}^{h-2}}{2k+h+n} \Pi_{k+2,0}(x_ix_j p) .
\end{equation}
To deduce the last formula, we have implicitly used that if $q_1,q_2$ are two polynomials and $q_2$ is $k$-homogeneous, then
\begin{equation*}
    \Pi_{k',0}\big(q_1\,\Pi_{k, 0}(q_2)\big) = \Pi_{k', 0}(q_1\,q_2) .
\end{equation*}
Using again the latter identity, \cref{eq:loc43} implies
\begin{align*}
    \pi_{k', h-2}(Tf) 
    &= \pi_{k',h-2}\Big(\sum_{i,j}(\delta_{ij}-A_{ij})\pi_{k+2,h-2}\big(\partial_{ij}\lapl^{-1}f\big)\Big) \\
    &=
    \frac{h\abs{x}^{h-2}}{2k+h+n}\Pi_{k',0}\Big(\sum_{i,j}(\delta_{ij}-A_{ij}(x))x_ix_j p\Big),
\end{align*}
which coincides with \cref{eq:tmp44}.
\end{proof}

\begin{proof}[Proof of \cref{thm:precisedenominators}]
We consider only the case $\lambda < \infty$ as the case $\lambda=\infty$ is substantially simpler and can be proven by appropriately adapting the arguments.

For any $\ell\ge0$, let $\widetilde \Sigma_\ell$ be defined as (compare with $\Sigma_\ell$ defined in \cref{lem:controlG1G2})
\begin{equation*}
    \widetilde\Sigma_\ell\defeq
    \{(x,y)\in\Z^2:\, x+y\ge\ell,\, y\ge -2\ell,\, 4x+3y\ge 0,\, 2x+\lambda y\ge 0,\, y\text{ is even}\} .
\end{equation*}
Notice that $\widetilde\Sigma_\ell=\widetilde\Sigma_1+\widetilde\Sigma_1+\cdots+\widetilde\Sigma_1$, where the sum is repeated $\ell$ times.

Given a homogeneous polynomial $p\in\H^k$, thanks to the definition of $\lambda$, we have $\Pi_{k',0}(\alpha\cdot p)=0$  if $k'<\lambda + k$ and $\Pi_{\lambda+k,0}(\alpha\cdot p)\not=0$ if $p\not=0$, where the projections $\Pi_{i,j}$ are defined in \cref{thm:harmonicdecomposition}.
Hence, for a nonzero function $f\in\H^{k,h}$, we deduce from \cref{lem:lowest_arrow} that $\pi_{k',h'}(Tf)=0$ whenever $(k',h')\not\in(k,h) + \widetilde\Sigma_1$ and $\pi_{k+\lambda, h-2}(Tf)\not=0$ if $h<0$. For the $\delta$ distribution, we have that $\pi_{k',h'}(T\delta)=0$ whenever $(k',h')\not\in (0,-n)+\widetilde\Sigma_1$ and $\pi_{\lambda,-n-2}(T\delta)\not=0$.

Hence, it is not hard to prove by induction on $\ell\ge 1$ that $\pi_{k,h}(T^\ell\delta)=0$ whenever $(k,h)\not\in(0,-n)+\widetilde\Sigma_\ell$, and $\pi_{\lambda\ell', -n-2\ell'}(T^\ell\delta)\not=0$ if and only if $\ell'=\ell$.

Recalling the definition of the fundamental solution $u(x)=\lapl^{-1}\big(\delta + T\delta + T^2\delta + \cdots\big)$ and the properties of the inverse Laplacian \cref{lem:inverse_lapl}, we deduce (notice that $h=-2\frac{k+h}{\lambda-2}$ is equivalent to $2k+\lambda h=0$)
\begin{equation}\label{eq:tmp551}\begin{aligned}
    \pi_{k, h-(n-2)}(u) = 0 \quad&\text{if $h<-2\frac{k+h}{\lambda-2}$}, \\
    \pi_{k, h-(n-2)}(u) \not=0 \quad&\text{if $h = -2\frac{k+h}{\lambda-2}$}.
\end{aligned}\end{equation}
Since 
\begin{equation*}
    u_\ell(x) = \abs{x}^{n-2}\sum_{k+h=\ell} \pi_{k,h-(n-2)}(u) ,
\end{equation*}
the proof is concluded because \cref{eq:tmp551} is equivalent to the desired statement.
\end{proof}

\printbibliography[heading=bibintoc]
\end{document}